\documentclass[12pt]{article}%
\usepackage{amsmath}
\usepackage{amsfonts}
\usepackage{amssymb}
\usepackage{graphicx}%
\usepackage[shortlabels]{enumitem}
\usepackage{amsthm}
\usepackage{parskip}
\usepackage{mathrsfs}
\usepackage{pgf,tikz,pgfplots}
\usetikzlibrary{patterns}

\usepackage{mathtools}
\usepackage{url}
\usepackage[authoryear,round]{natbib}
\newtheorem{theorem}{Theorem}[section]

\newtheorem{lemma}[theorem]{Lemma}

\newtheorem*{corollary*}{Corollary}

\theoremstyle{definition}
\newtheorem{example}[theorem]{Example}

\makeatletter
\def\thm@space@setup{%
  \thm@preskip=\parskip \thm@postskip=0pt
}
\makeatother

\newenvironment{proof*}[1][Proof]{\noindent\textbf{#1.} }{\ \rule{0.5em}{0.5em}}

\newcommand{\rr}{\mathbb{R}}

\newcommand{\dir}{\mathrm{Dir}}
\newcommand{\bb}[1]{\boldsymbol{#1}}
\newcommand{\ee}[2]{{#1}\times 10^{#2}}

\def \hfillx {\hspace*{-\textwidth} \hfill}

\DeclareMathOperator{\vol}{Vol}

\DeclarePairedDelimiter\norm{\lVert}{\rVert}

\numberwithin{equation}{section}

\newcommand{\blind}{1}

\begin{document}

\def\spacingset#1{\renewcommand{\baselinestretch}%
{#1}\small\normalsize} \spacingset{1}


\if1\blind
{
  \title{\bf Superiority of Bayes estimators over the MLE in high dimensional multinomial models and its implication for nonparametric Bayes theory}
  \author{Rabi Bhattacharya\thanks{
    The authors gratefully acknowledge \textit{ NSF grant DMS 1811317.}}\hspace{.2cm}\\
    Department of Mathematics, University of Arizona\\
    and \\
    Rachel Oliver\footnotemark[1] \\
    Department of Mathematics, University of Arizona}
  \maketitle
} \fi

\if0\blind
{
  \bigskip
  \bigskip
  \bigskip
  \begin{center}
    {\LARGE\bf  Superiority of Bayes estimators over the MLE in high dimensional multinomial models and its implication for nonparametric Bayes theory}
\end{center}
  \medskip
} \fi

\bigskip
\begin{abstract}
This article focuses on the performance of Bayes estimators, in comparison with the MLE, in multinomial models with a relatively large number of cells. The prior for the Bayes estimator is taken to be the conjugate Dirichlet, i.e., the multivariate Beta, with exchangeable distributions over the coordinates, including the non-informative uniform distribution. The choice of the multinomial is motivated by its many applications in business and industry, but also by its use in providing a simple nonparametric estimator of an unknown distribution. It is striking that the Bayes procedure outperforms the asymptotically efficient MLE over most of the parameter spaces for even moderately large dimensional parameter space and rather large sample sizes.
\end{abstract}

\noindent%
\vfill

\newpage
\spacingset{1.5} 


\section{Introduction}

The present article shows by analytical computations and simulations that Bayes estimators even in moderately high dimensional multinomial models outperform the MLE on most of the parameter space. High dimensional multinomial models are useful for industrial planning. For example, for planning its manufacturing process a big departmental store may try to estimate the proportions of a certain type of clothing, by sizes and/or colors, demanded by its customers (see the example in Section \ref{sec:data}).

For another important motivation for pursuing this study of multinomials, note that an unknown distribution on a state space may be approximated nonparametrically by probabilities of members of a fine partition. Sampling from this approximate distribution is the same as sampling from a multinomial. Bayes estimation of the approximating multinomial with a conjugate Beta (i.e. Dirichlet), prior was a motivation for Ferguson’s path breaking development of nonparametric Bayes theory of estimation of a probability distribution on the state space with the so called Dirichlet process prior \citep{ferguson1973}.

It may be pointed out that the venerable Bernstein-von Mises theorem (see, e.g., \citet[p.~339]{bickel2001}, or \citet[p.~190]{textbook}) is designed to show how Bayes estimators with reasonable priors are asymptotically as efficient as the MLE and approach it as the sample size increases. It turns out from our study that in high dimensional multinomial models it is the MLE which tries to catch up with Bayes as the sample size increases!

The MLE of a multinomial with parameter $\bb\theta=(\theta_1,\theta_2,\dotsc,\theta_k)$ is given by $\bb{\hat\theta}=(\hat\theta_1,\hat\theta_2,\dotsc,\hat\theta_k)$, where $\hat\theta_i=n_i/n$ is the proportion of i.i.d. observations $X_1,\dotsc,X_n$ in the $i$-th cell estimating its true proportion $\theta_i$. The Bayes estimator $\bb{d_B}$ uses the conjugate prior $\operatorname{Beta}(\alpha_1,\alpha_2,\dotsc,\alpha_k)$, also called a Dirichlet prior and denoted $\dir(\alpha_1,\alpha_2,\dotsc,\alpha_k)$. We take $\alpha_1=\alpha_2=\dotsb=\alpha_k$, for purposes of ease in computation and assuming no a priori preference for some categories over others. Indeed, when the common value is 1 one has the so-called \emph{non-informative prior}, assigning the uniform density on the parameter space

\begin{equation}
E_k\equiv \left\{\bb\theta = (\theta_1,\theta_2,\dotsc,\theta_k): \theta_i \ge 0 \;\forall i, \sum_{1\le i \le k} \theta_i =1 \right\}.
\label{eq:ek1}
\end{equation}

The posterior distribution of $\bb\theta$ is $\dir(\alpha_1+n_1,\dotsc,\alpha_k+n_k)$ and its mean is the Bayes estimator $\bb{d_B}$ of $\bb\theta$, as given in \eqref{eq:db}.

For the most part, we take the loss function to be squared error under which the risk functions $R(\bb{\hat\theta},\bb\theta)$ and $R(\bb{d_B},\bb\theta)$ are given by \eqref{eq:mlerisk} and \eqref{eq:Bayesrisk}. Both analytical calculations and simulations are carried out for cases with $\alpha_1=\alpha_2=\dotsb=\alpha_k$.

For an approximation of the nonparametric estimation of an unknown probability on some state space, the Bayes procedure with a Dirichlet prior and a base measure $\bb\alpha$ with total mass $C$, we consider for each $k$, $\alpha_i\equiv\bb\alpha(\{i\})=C/k \;\forall i 1\le i \le k$. Such an approximation with large $k$, named the ``tree construction'' of the Dirichlet process may be found in \citet[Chapter 3]{ghosh_ram2003} and \citet[Chapter 4]{ghosal2017}. One may think of $C/k$ as the $\bb\alpha$ measure of each of $k$ members of a partition, which can be ensured if $\bb\alpha$ is absolutely continuous (with respect to Lebesgue measure on an Euclidean state space or a volume measure on a manifold).

Section \ref{sec:mult} introduces the multinomial distribution and the estimators under consideration. Section \ref{sec:volume} explores the volume of the parameter space where the Bayes estimators have lower risk than the MLE. To build intuition we begin in Sections \ref{sub:omega2} and \ref{sub:omega3} with analytical computations for the binomial model, i.e., a multinomial model with $k=2$, and then a multinomial with $k=3$. For $k\ge 4$, the geometry of the region of the simplex \eqref{eq:ek1} where $R(\bb{\hat\theta},\bb\theta)\le R(\bb{d_B},\bb\theta)$ is complex. Extensive simulations show that, under the uniform prior $\dir(1,1,\dotsc,1)$, for moderate and large vaules of $k$ and small as well as large values of the sample size $n$, the Bayes estimator has a smaller expected squared error than that of the MLE on most of the parameter space (see Figure \ref{fig:sampling}).

Next consider the multivariate Beta prior with parameters $\alpha_1=\alpha_2=\dotsb=\alpha_k=1/k$ with a large $k$, suitable for a simple nonparametric estimation of an unknown distribution as alluded to above. In this case, for sufficiently large $k$, the region where the MLE has a smaller expected squared error than Bayes is identifiable as the union of $k$ regions, each being a ``cone'' shaped structure minus a cap at the base (see Figure \ref{fig:delta3} in Section \ref{sec:volume} for the case $k=3$). Its area, i.e., its volume measure in the simplex $E_k$ in \eqref{eq:ek1}, relative to the volume measure of $E_k$ or, equivalently, of its complement, is estimated analytically in Section 3 (see Lemma \ref{lem:sa}). This conservative estimate is compared with the ``true'' value obtained by simulation in Table \ref{tab:1kprior}, once again showing that the region where the MLE has a smaller expected risk that that of the Bayes is rather negligible.

In Section \ref{sec:avg}, the average difference over the parameter space of expected squared errors $R(\bb{\hat\theta},\bb\theta)-R(\bb{d_B},\bb\theta)$ is computed in proportion to the average of $R(\bb{\hat\theta},\bb\theta)$ and compared (see Table \ref{tab:comparison} and Figure \ref{fig:avgriskdec}). This scaled difference is shown to be maximum under the uniform prior$\dir(1,1,\dotsc,1)$. This may come as a suprise because the proportion of the volume of $E_k$ in which the Bayes risk is smaller than that of the MLE is generally larger under the prior $\dir(1/k,\dotsc,1/k)$ than under the uniform!

Section \ref{sec:tv} briefly illustrates the approximations in $L_1$ or total variation distance between a true distribution and its nonparametric estimators based on the MLE and those of the nonparametric Bayes estimators. Section \ref{sec:data} presents a data example to illustrate an industrial application of a high dimensional multinomial. A final Section \ref{sec:final} lays down some final Remarks.


\section{The multinomial distribution}\label{sec:mult}

Consider the estimation of the parameter $\boldsymbol \theta=(\theta_1,\theta_2,\dotsc,\theta_k)$ in the multinomial distribution, where $\theta_j$ is the proportion of the $j$-th class in a population with $k\ge 2$ classes ($j=1,2,\dotsc,k$). Based on a simple random sample of size $n$ from the population, let $n_1,n_2,\dotsc,n_k$ be the numbers in the sample belonging to each of the $k$ classes. Since $(n_1,n_2,\dotsc,n_k)$ is a sufficient statistic for $\boldsymbol\theta$, it is enough to consider the distribution of $(n_1,n_2,\dotsc,n_k)$ for the estimation of $\boldsymbol\theta$. Namely,

\begin{equation}
f(n_1,n_2,\dotsc,n_k;\boldsymbol\theta)=\frac{n!}{n_1!n_2!\dotsb n_k!}\theta_1^{n_1}\theta_2^{n_2}\dotsb \theta_k^{n_k}\quad\quad 
\left\{(\theta_1,\theta_2,\dotsc,\theta_k)\in\rr^k: \theta_j\ge 0~\forall j, \sum_{j=1}^k \theta_j=1\right\}.
\label{eq:mult}
\end{equation}

The Maximuim Likelihood Estimator (MLE) is $\hat\theta\equiv\left(\frac{n_1}{n},\frac{n_2}{n},\dotsc,\frac{n_k}{n}\right)$. The multivariate Beta, or \emph{Dirichlet}, prior $\dir(\alpha_1,\alpha_2,\dotsc,\alpha_k)$ has density with respect to Lebesgue measure on $\Theta^\sim$, where $\Theta^\sim$ is given by 

\[
\Theta^\sim\equiv
\left\{ (\theta_1,\theta_2,\dotsc,\theta_{k-1})\in\rr^{k-1}: \theta_j\ge 0~\forall j, \sum_{j=1}^{k-1} \theta_j\le1\right\}.
\]

The Dirichlet density is

\begin{equation}
\pi(\theta_1,\theta_1,\dotsc,\theta_k)=\frac{\Gamma(\alpha_1+\dotsb+\alpha_k)}{\Gamma(\alpha_1)\Gamma(\alpha_2)\dotsb\Gamma(\alpha_k)} \theta_1^{\alpha_1-1}\theta_2^{\alpha_2-1}\dotsb\theta_k^{\alpha_k-1},\quad \text{for }\boldsymbol\theta\in\Theta^\sim, 
\label{eq:dir}
\end{equation}

where  $\theta_k=1-\theta_1-\theta_2-\dotsb-\theta_{k-1}$.

It is well known , and easy to prove, that if the prior is $\dir(\alpha_1,\dotsc,\alpha_k)$, the posterior distribution of $\boldsymbol\theta$ is Dirichlet $\dir(\alpha_1+n_1,\alpha_2+n_2,\dotsc,\alpha_k+n_k)$ (see, e.g., \citet{textbook}).

Under squared error loss: $L(\boldsymbol\theta,\boldsymbol\theta')=|\boldsymbol\theta-\boldsymbol\theta'|^2=\sum_{1\le i\le k} (\theta_i-\theta_i')^2$, then the risk function of the MLE ($\hat\theta=(\hat\theta_1,\dotsc,\hat\theta_k)$ with $\hat\theta_i=n_i/n$) is given by

\begin{equation}
R(\boldsymbol{\hat\theta},\boldsymbol\theta)=\sum_{1\le i \le k} \frac{\theta_i(1-\theta_i)}{n}=\frac{1-\sum_{1\le i\le k} \theta_i^2}{n}
\label{eq:mlerisk}
\end{equation}

We wish to choose an exchangeable prior--invariant under permutation of coordinates. Thus we choose $\alpha_1=\alpha_2=\dotsb =\alpha_k=C_{k,n}$, where $C_{k,n}$ is some constant which may depend on $k$ and $n$. The choices of $C_{k,n}$ that lead to better estimators in terms of risk under squared error loss will be investigated.

Under the Dirichlet prior $\dir(\alpha_1,\dotsc,\alpha_k)$ with $\alpha_1=\alpha_2=\dotsb=\alpha_k=C_{k,n}$ and squared error loss, the Bayes estimator is 

\begin{equation}
\boldsymbol d_B=(d_{B1},\dotsc,d_{Bk}), \quad \text{with } d_{Bi}=\frac{n_i+C_{k,n}}{n+kC_{k,n}} \quad (i=1,2,\dotsc, k)
\label{eq:db}
\end{equation}

and its risk function is (see, e.g., \citet{textbook})

\begin{align}
R(\boldsymbol d_B, \boldsymbol\theta)&=\sum_{1\le i \le k} \frac{n\theta_i(1-\theta_i)+(C_{k,n}-k\theta_iC_{k,n})^2}{\left(n+kC_{k,n}\right)^2}\nonumber \\ 
&=\left(\sum_{1\le i \le k} (\theta_i-\theta_i^2)\right)\frac{n}{\left(n+kC_{k,n}\right)^2} + \frac{C_{k,n}^2\left(k-2k+k^2\left(\sum_{1\le i \le k} \theta_i^2\right)\right)}{\left(n+kC_{k,n}\right)^2}\nonumber  \\ 
&=\left( 1- \sum_{1\le i \le k} \theta_i^2\right)\frac{n}{\left(n+kC_{k,n}\right)^2} - \frac{kC_{k,n}^2}{\left(n+kC_{k,n}\right)^2}+\frac{k^2C_{k,n}^2 \sum_{1\le i \le k} \theta_i^2}{\left(n+kC_{k,n}\right)^2} \nonumber \\
&=\frac{n-kC_{k,n}^2}{\left(n+kC_{k,n}\right)^2}+\frac{\left(k^2C_{k,n}^2-n\right) \sum_{1\le i \le k} \theta_i^2}{\left(n+kC_{k,n}\right)^2}
\label{eq:Bayesrisk}
\end{align}

Hence, $R(\boldsymbol{\hat\theta},\boldsymbol\theta)\le R(\boldsymbol d_B,\boldsymbol\theta)$ (and thus the MLE has lower risk than the Bayes estimator) only on the set

\begin{equation}
\left\{ \boldsymbol\theta = (\theta_1,\theta_2,\dotsc,\theta_k): \theta_i \ge 0~\forall i, \sum_{1 \le i \le k} \theta_i=1, {\left[\frac{k^2C_{k,n}^2-n}{\left(n+kC_{k,n}^2\right)^2}+\frac{1}{n}\right] \sum_{1\le i \le k} \theta_i^2 \ge \frac{1}{n}-\frac{n-kC_{k,n}^2}{\left(n+kC_{k,n}\right)^2}}\right\}.
\end{equation}

This can be written more compactly as 

\begin{equation}
\left\{ \boldsymbol\theta = (\theta_1,\theta_2,\dotsc,\theta_k): \theta_i \ge 0~\forall i, \sum_{1 \le i \le k} \theta_i=1, { \sum_{1\le i \le k} \theta_i^2 \ge \frac{2n+(n+k)C_{k,n}}{2n+(k+kn)C_{k,n}}}\right\}.
\label{eq:riskset}
\end{equation}

 Recall that the parameter space is the simplex $E_k$

\begin{equation}
E_k\equiv \{ \boldsymbol \theta=(\theta_1,\theta_2,\dotsc,\theta_k): \theta_i\ge 0~\forall i, \sum_{1 \le i \le k} \theta_i=1\}
\label{def:ek}
\end{equation}

We will calculate/simulate the volume of the region \ref{eq:riskset} in $E_k$ for various choices of $C_{k,n}$. This gives the proportion of the parameter space that is better estimated (with regard to risk) by the MLE. 

It is seen (see Section 4) that even for fairly large sample sizes the Bayes estimator outperforms the MLE after $k$ is large, such as $k\ge 10$.


\section{The proportion of the parameter space favoring the Bayes estimator: volume calculations} \label{sec:volume}

Before calculating the volume of the region \eqref{eq:riskset}, we will consider, more generally, the simplex $E_k$ defined in \eqref{def:ek}, and the region $\Omega_{k,R}\equiv \{ \boldsymbol\theta \in E_k: |\boldsymbol\theta|^2\le R\}$, where $R$ is a known constant. Note that the region \eqref{eq:riskset} is the complement of $\Omega_{k,R}$ (for a specific choice of $R$). Thus the region $\Omega_{k,R}$, when applied to the this problem, represents the region of the parameter space where the Bayes estimator has lower risk than the MLE. 

Define the point $\boldsymbol{e_0}$ by

\begin{equation}
\boldsymbol{e_0}\equiv \left(\frac{1}{k},\frac{1}{k},\dotsc,\frac{1}{k}\right),
\label{eq:e0}
\end{equation}

 the point in $E_k$ that is closest to the origin, with $\left|\boldsymbol{e_0}\right|^2=\frac{1}{k}$. We can see, then, that if $R < 1/k$ then $\Omega_{k,R}=\O$. Similarly, if $R \ge 1$ then $\Omega_{k,R}=E_k$, since $|\boldsymbol\theta|^2\le1~\forall \theta\in E_k$.

For $R\ge 1/k$, define $\delta_k(R)$ to be the distance between $\boldsymbol{e_0}$ and the sphere $\{ \boldsymbol\theta\in E_k: |\boldsymbol\theta|^2=R\}$. Then

\begin{equation}
\delta_k(R)=\sqrt{R-\frac{1}{k}}.
\label{eq:delta}
\end{equation}

Define $\nu_j$ to be the distance between $\boldsymbol{e_0}$ and the $(k-1-j)$-dimensional boundary of $E_k$. This is the distance between $\boldsymbol{e_0}$ and $\left(0,\dotsc,0,\frac{1}{k-j},\dotsc, \frac{1}{k-j}\right)$, which has the first $j$ coordinates equal to 0, and the remaining $k-j$ coordinates equal to $\frac{1}{k-j}$. We have

\begin{equation}
\nu_j=\sqrt{\frac{j}{k(k-j)}}.
\label{eq:nu}
\end{equation}

Note that we take $j=1,\dotsc,k-1$ and that $\nu_1<\nu_2<\dotsb<\nu_{k-1}$. Thus, for any $R\in \left(1/k,1\right)$, we can find $j$ such that $\nu_j\le \delta_k(R) < \nu_{j+1}$. We conjecture that the precise shape of $\Omega_{k,R}$, and thus the formula for calculating its surface area, should depend on which $j$ satifies this condition. We use the term ``surface area'' since $\Omega_{k,R}$ (and also $E_k$) is a $(k-1)$-dimensional subspace of $\rr^k$.

\subsection{The surface area of $\boldsymbol{E_k}$}

Consider in general the simplex $S_k(r)$, defined

\[
S_k(r)\equiv \{\boldsymbol\theta = (\theta_1,\theta_2,\dotsc,\theta_k):~ \theta_i\ge0~\forall i, \sum_{1\le i\le k}\theta_i \le r\}, ~r>0.
\]

Let $E_k(r)$ be the boundary of $S_k(r)$, namely,

\[
E_k(r)\equiv \{\boldsymbol\theta = (\theta_1,\theta_2,\dotsc,\theta_k):~ \theta_i\ge0~\forall i, \sum_{1\le i\le k}\theta_i = r\},
\]

and write $E_k\equiv E_k(1)$.

\begin{lemma}

(i) The volume $V_k(r)$ of $S_k(r)$ is $\frac{r^k}{k!}$, and (ii) the surface area $A_k(r)$ of $E_k(r)$ is $r^{k-1}\frac{\sqrt{k}}{(k-1)!}$.

In particular, the surface area of the simplex $E_k\equiv E_k(1)$ is 

\begin{equation}
A_k\left(\equiv A_k(1)\right)=\frac{\sqrt{k}}{(k-1)!}.
\label{eq:ekvol}
\end{equation}

\end{lemma}

\begin{proof}
\begin{enumerate}[(i)]
\item 

\begin{align*}
V_k(r)&=\int_{S_k(r)}d\theta_1d\theta_2\dotsb d\theta_k\\
&=\int_{S_{k-1}(r)} \left(r-\sum_{1\le i\le k-1} \theta_i\right)d\theta_1 d\theta_2\dotsb d\theta_{k-1}\\ 
&=\int_{S_{k-2}(r)} \frac{\left(r-\sum_{1\le i\le k-2} \theta_i\right)^2}{2} d\theta_1 d\theta_2\dotsb d\theta_{k-2}\\
&=\dotsb\\
&= \int_{S_1(r)} \frac{(r-\theta_1)^{k-1}}{(k-1)!} d\theta_1\\
&=\frac{r^k}{k!}. 
\end{align*}

\item The difference in volume between $S_k(r)$ and $S_k(r+\Delta r)$ is a slab around $E_k(r)$. Note that

\begin{align*}
V_k(r+\Delta r)-V_k(r)&=\Delta r \left[\frac{d}{dr} \left(\frac{r^k}{k!}\right)\right]+o(\Delta r)\\
&= \frac{(\Delta r) r^{k-1}}{(k-1)!} \quad\text{as } \Delta r\searrow 0.\\
\end{align*}

The unit normal to the surface $E_k(r)$  at every point on it is $(\text{grad} \sum_{1\le i\le k}\theta_i)/|\text{grad} \sum_{1\le i\le k}\theta_i|=\left(\frac{1}{\sqrt{k}},\frac{1}{\sqrt{k}},\dotsc,\frac{1}{\sqrt{k}}\right)$. Hence at every point the thickness of the slab $S_k(r+\Delta r) \setminus S_k(r)$ is $\Delta r/\sqrt{k}$. One may also see this by computing the distance between $E_k(r)$ and $E_k(r+\Delta r)$ along the normal through the origin, i.e. $\left|\left(\frac{r}{k},\frac{r}{k},\dotsc,\frac{r}{k}\right)-\left(\frac{r+\Delta r}{k},\frac{r+\Delta r}{k},\dotsc,\frac{r+\Delta r}{k}\right)\right|$.

The surface area $A_k(r)$ then is given by

\begin{align*}
A_k(r)&=\lim_{\Delta r\to 0} \frac{V_k(r+\Delta r) - V_k(r)}{\Delta r/\sqrt{k}}\\
&=r^{k-1}\frac{\sqrt{k}}{(k-1)!}~.\\
\end{align*}
\end{enumerate}
\end{proof}

\subsection{The Surface Area of $\boldsymbol{\Omega_{2,R}}$}\label{sub:omega2}

Let us calculate the $k=2$ case (this corresponds to the binomial distribution, a special case of the multinomial distribution with $k=2$). The simplex $E_2=\{\boldsymbol\theta\in \rr^2:~ \theta_1,\theta_2\ge 0,~ \theta_1+\theta_2=1\}$ is the line between $(0,1)$ and $(1,0)$. The region of interest is $\Omega_{2,R}=\{\boldsymbol\theta\in E_2:~ \theta_1^2+\theta_2^2\le R\}$. See Figure \ref{fig:k2} for an illustration of this region.

\begin{figure}[h!]
\centering

\begin{tikzpicture}[scale=4]

\draw[thick, ->] (0,0) -- (1.25,0) node[anchor=north west] {$\theta_1$};
\draw[thick, ->] (0,0) -- (0,1.25) node[anchor=south east] {$\theta_2$};
\foreach \x in {0,1}
   \draw (\x cm,1pt) -- (\x cm,-1pt) node[anchor=north] {$\x$};
\foreach \y in {0,1}
    \draw (1pt,\y cm) -- (-1pt,\y cm) node[anchor=east] {$\y$};
\draw[->] (0,0) -- ({.49 + .49*sqrt(2*0.8^2-1)},{.49 - .49*sqrt(2*0.8^2-1)}) node[below= 3pt, left=23pt] {$\sqrt{R}$};
\draw[very thick] (0,1) -- (1,0);
\draw[very thick, dashed] (0.8,0) arc (0:90:0.8 cm);
\draw[line width = 2pt, yellow, opacity=.5] ({.5 - .5*sqrt(2*0.8^2-1)},{.5 + .5*sqrt(2*0.8^2-1)}) -- ({.5 + .5*sqrt(2*0.8^2-1)},{.5 - .5*sqrt(2*0.8^2-1)});
\filldraw ({.5 - .5*sqrt(2*0.8^2-1)},{.5 + .5*sqrt(2*0.8^2-1)}) circle (.5pt) node[above right] {$p_1$};
\filldraw ({.5 + .5*sqrt(2*0.8^2-1)},{.5 - .5*sqrt(2*0.8^2-1)}) circle (.5pt) node[above right] {$p_2$};
\end{tikzpicture}

\caption{An illustration of $E_2$. The region $\Omega_{2,R}$ is the line sement between $p_1$ and $p_2$.}
\label{fig:k2}
\end{figure}
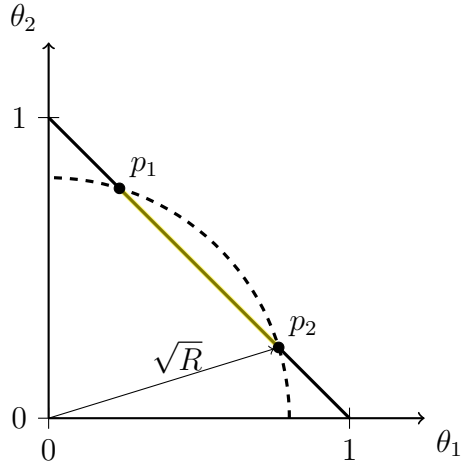

We find the intersection points by solving $\theta_1^2+(1-\theta_1)^2=R$ to obtain the points 

\[
p_1=\left(\frac{1}{2}-\frac{1}{2}\sqrt{2R-1},\frac{1}{2}+\frac{1}{2}\sqrt{2R-1}\right)
\]

and 

\[
p_2=\left(\frac{1}{2}+\frac{1}{2}\sqrt{2R-1},\frac{1}{2}-\frac{1}{2}\sqrt{2R-1}\right).
\]
 
The surface area of $\Omega_{2,R}$ is the length of the line segment between these two points, which is $\sqrt{4R-2}$. 

We can also find $\delta_2(R)=\sqrt{R-1/2}$ and $\nu_1=\sqrt{1/2}$ using equations \eqref{eq:delta} and \eqref{eq:nu}, respectively. These distances are pictured in Figure \ref{fig:delta2}. Note that we also have that the 1-dimensional volume of $\Omega_{2,R}$ is equal to $2\delta_2(R)=2\sqrt{R-1/2}=\sqrt{4R-1}$.

\begin{figure}[h!]
\centering

\begin{tikzpicture}[scale=4]

\draw[thick, ->] (0,0) -- (1.25,0) node[anchor=north west] {$\theta_1$};
\draw[thick, ->] (0,0) -- (0,1.25) node[anchor=south east] {$\theta_2$};
\foreach \x in {0,1}
   \draw (\x cm,1pt) -- (\x cm,-1pt) node[anchor=north] {$\x$};
\foreach \y in {0,1}
    \draw (1pt,\y cm) -- (-1pt,\y cm) node[anchor=east] {$\y$};
\draw[very thick] (0,1) -- (1,0);
\draw[very thick, dashed] (0.8,0) arc (0:90:0.8 cm);
\draw[line width = 2pt, yellow, opacity=.5] ({.5 - .5*sqrt(2*0.8^2-1)},{.5 + .5*sqrt(2*0.8^2-1)}) -- ({.5 + .5*sqrt(2*0.8^2-1)},{.5 - .5*sqrt(2*0.8^2-1)});
\filldraw (.5,.5) circle (.5pt) node[below left] {$e_0$};
\draw[<->, blue] ({.55 - .5*sqrt(2*0.8^2-1)},{.53 + .5*sqrt(2*0.8^2-1)}) -- node[above right] {$\delta_2(R)$} (.55,.53);
\draw[<->, red] (.55,.53) -- node[above right] {$\nu_1$} (.99,.09);
\end{tikzpicture}

\caption{An illustration of $\Omega_{2,R}$ with the distances $\delta_2(R)$ and $\nu_1$ labeled.}
\label{fig:delta2}
\end{figure}
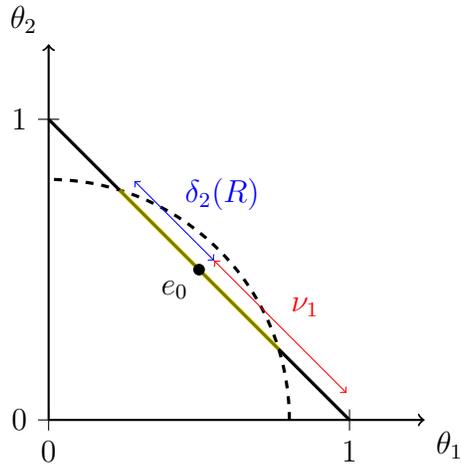

The length of the line $E_2$ is $\sqrt{2}$, giving that the proportion of the 1-dimensional volume of $ E_2$ made up by $\Omega_{2,R}$ is 

\begin{align*}
\frac{\vol(\Omega_{2,R})}{\vol(E_2)}&=\frac{\sqrt{4R-2}}{\sqrt{2}}\\
&=\sqrt{2R-1}
\end{align*}

\subsection{The surface area of $\boldsymbol{\Omega_{3,R}}$}\label{sub:omega3}

For $k=3$, we can also calculate this volume exactly. The simplex $E_3=\{\boldsymbol\theta\in \rr^3:~ \theta_1,\theta_2, \theta_3\ge 0,~ \theta_1+\theta_2+\theta_3=1\}$ is an equilateral triangle between the points $(1,0,0)$, $(0,1,0)$, and $(0,0,1)$. See Figure \ref{fig:smallrad} for an illustration of the space for $R\in \left(\frac{1}{3},\frac{1}{2}\right)$ and Figure \ref{fig:largerad} for an illustration of the space for $R\in \left(\frac{1}{2},1\right)$.

\begin{figure}[h]
\centering
\includegraphics[width=.6\textwidth]{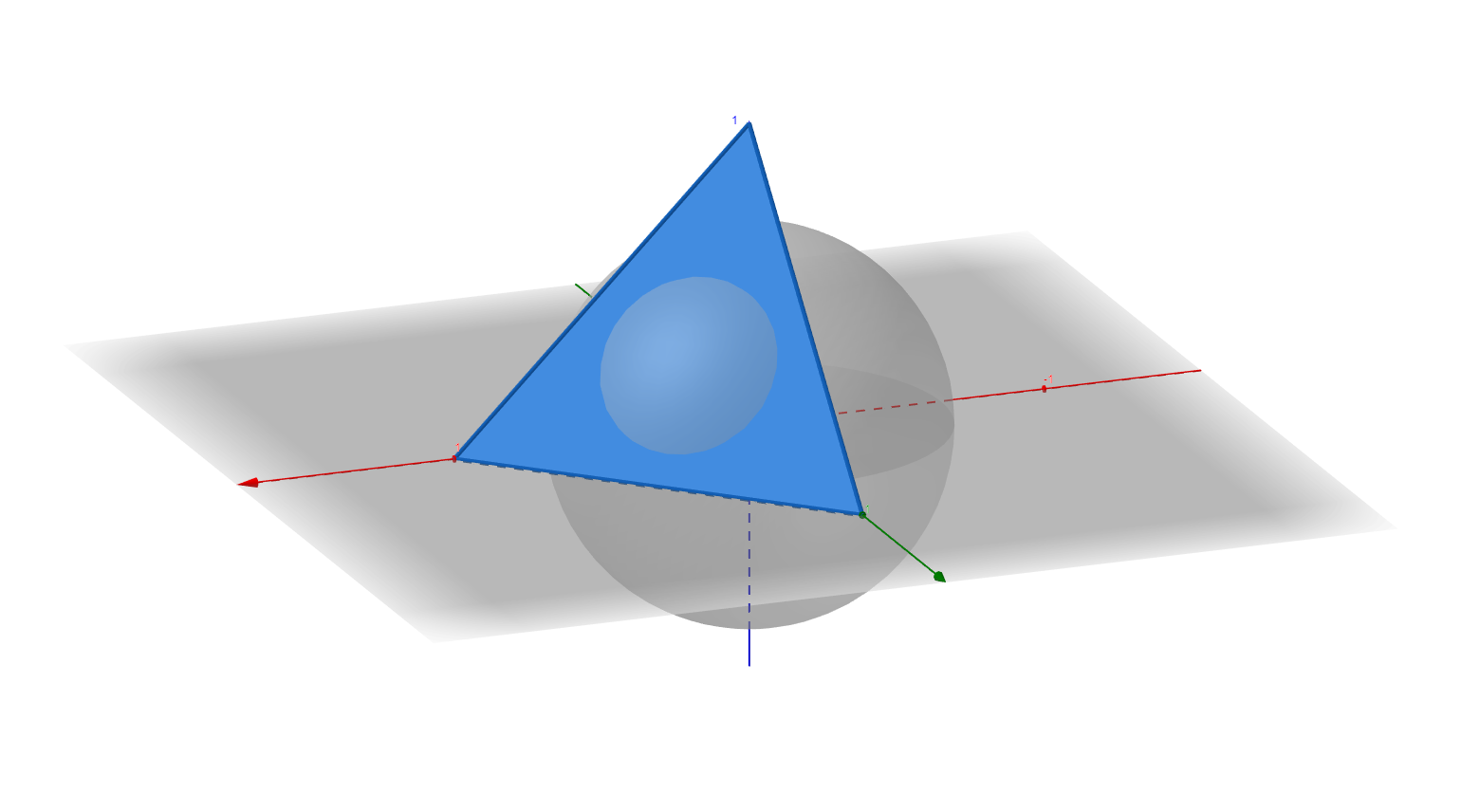}
\caption{An illustration of $E_3$ with the region $\Omega_{3,R}$ in gray for $R\in \left(\frac{1}{3},\frac{1}{2}\right)$.}
\label{fig:smallrad}
\end{figure}

\begin{figure}[h]
\centering
\includegraphics[width=.6\textwidth]{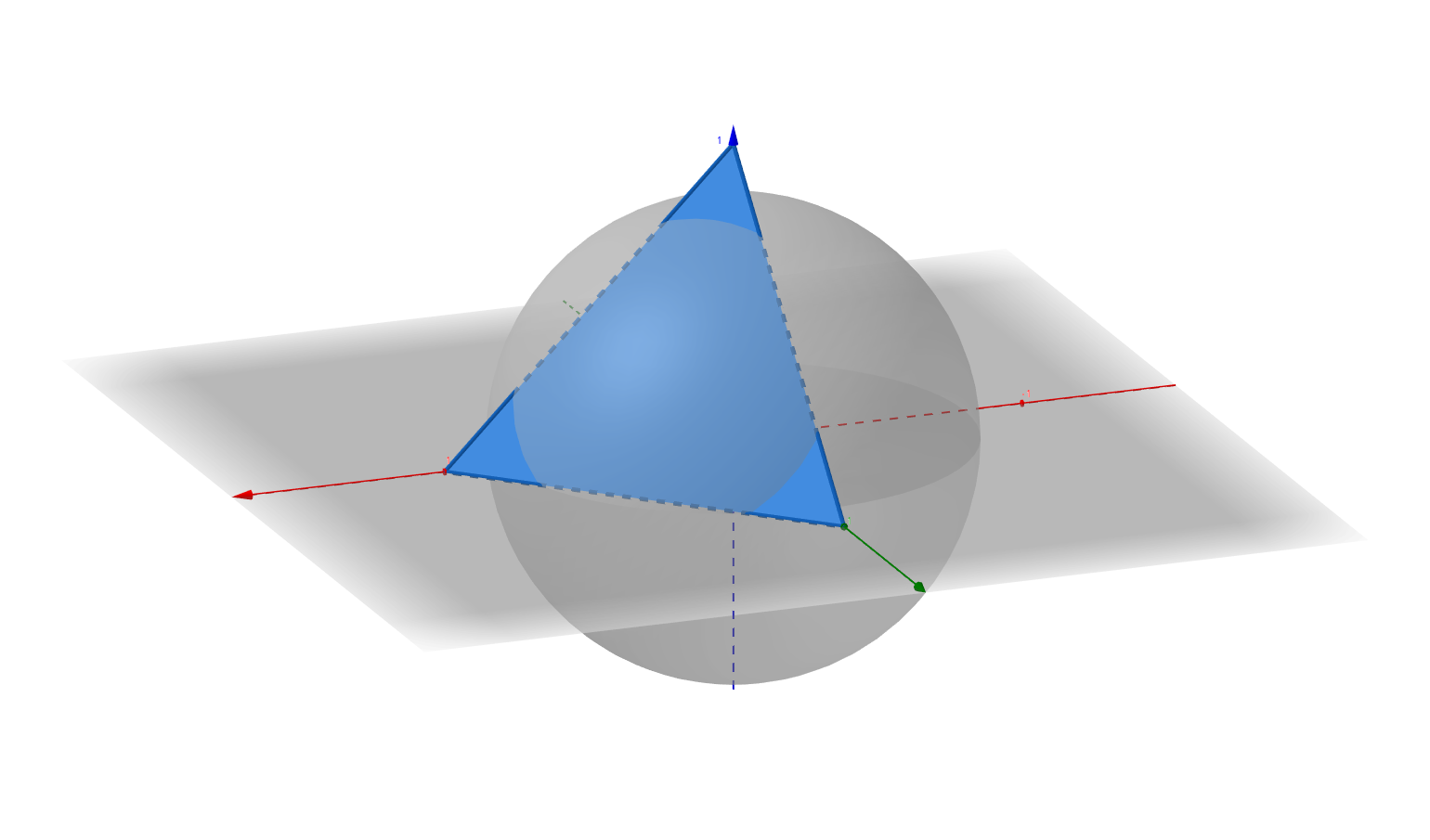}
\caption{An illustration of $E_3$ with the region $\Omega_{3,R}$ in gray for $R\in \left(\frac{1}{2},1\right)$.}
\label{fig:largerad}
\end{figure}

We calculate $\nu_1$ and $\nu_2$ using equation \eqref{eq:nu} and $\delta_3(R)$ using equation \eqref{eq:delta}. These are illustrated in Figure \ref{fig:delta3}.

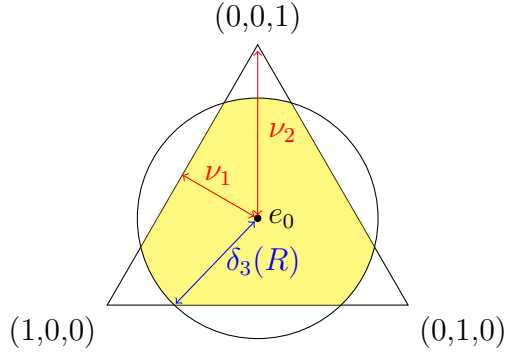
\begin{figure}[h]
\centering

\begin{tikzpicture}[scale=4]

\draw (0,0) node[below left] {(1,0,0)} -- (1,0) node[below right] {(0,1,0)} -- (.5, {sqrt(3)/2}) node[above] {(0,0,1)} -- cycle;

\begin{scope}
	\clip (0,0) -- (1,0) -- (.5, {sqrt(3)/2}) -- cycle;
	\fill[yellow, opacity = .5]   (.5, {sqrt(3)/6}) circle (.4cm);
\end{scope}
\draw (.5, {sqrt(3)/6}) circle (.4cm) ;
\filldraw (.5, {sqrt(3)/6}) circle (.3pt) node[right] {$e_0$} ;
\draw[<->,blue] (.49, {sqrt(3)/6-.05/sqrt(23)}) -- node[right] {$\delta_3(R)$}  ({(15-sqrt(69))/30+.001*sqrt(23)},0.005  );
\draw[<->,red] (.49, {sqrt(3)/6+sqrt(3)/300}) -- node[above] {$\nu_1$} (.25, {sqrt(3)/4});
\draw[<->,red] (.5, {sqrt(3)/6+.01}) -- node[right] {$\nu_2$} (.5, {sqrt(3)/2-.02});
\end{tikzpicture}

\caption{An illustration of $\Omega_{3,R}$ with $\nu_1 < \delta_3(R) < \nu_2$.}
\label{fig:delta3} 
\end{figure}

\begin{align*}
\nu_1&=\sqrt{\left(\frac{1}{3}\right)^2+2\left(\frac{1}{3}-\frac{1}{2}\right)^2}\\
&=\frac{1}{\sqrt{6}}~,\\
\nu_2&=\sqrt{2\left(\frac{1}{3}\right)^2+\left(\frac{1}{3}-1\right)^2}\\
&=\sqrt{\frac{2}{3}}~,
\end{align*}

and

\[
\delta_3(R)=\sqrt{R-\frac{1}{3}}~.
\]

If $\delta_3(R)\le \nu_1$, then $\Omega_{3,R}$ is just a circle with radius $\delta_3(R)$. Its surface area is then $\pi\left[\delta_3(R)\right]^2=\pi\left(R-\frac{1}{3}\right)$. The surface area of $E_3$ is (using equation \eqref{eq:ekvol}) $A_3=\sqrt{3}/2$. This gives that the proportion of the 2-dimensional volume of $E_3$ made up by $\Omega_{3,R}$ is 

\begin{align*}
\frac{\vol(\Omega_{3,R})}{\vol(E_3)}&=\frac{\pi\left(R-\frac{1}{3}\right)}{\sqrt{3}/2}\\
&=\frac{2\sqrt{3}}{3}\pi\left(R-\frac{1}{3}\right).\\
\end{align*}

If $\nu_1 < \delta_3(R) < \nu_2$ then we can divide up the region as in Figure \ref{fig:vol3} to determine the surface area.

\begin{figure}[h]
\centering

\begin{tikzpicture}[scale=4]

\draw (0,0) node[below left] {(1,0,0)} -- (1,0) node[below right] {(0,1,0)} -- (.5, {sqrt(3)/2}) node[above] {(0,0,1)} -- cycle;

\begin{scope}
	\clip (0,0) -- (1,0) -- (.5, {sqrt(3)/2}) -- cycle;
	\fill[yellow, opacity = .5]   (.5, {sqrt(3)/6}) circle (.4cm);
\end{scope}
\draw (.5, {sqrt(3)/6}) circle (.4cm) ;
\filldraw (.5, {sqrt(3)/6}) circle (.3pt) node[below, font=\tiny] {$e_0\equiv\left(\frac{1}{3},\frac{1}{3},\frac{1}{3}\right)$};
\draw[dashed, thin] (.5, {sqrt(3)/6}) -- ({(15-sqrt(69))/30},0  );
\draw[dashed, thin] (.5, {sqrt(3)/6}) -- ({(15+sqrt(69))/30},0  );
\draw[dashed, thin] (.5, {sqrt(3)/6}) -- ({(15-sqrt(69))/60},{sqrt(3)*(15-sqrt(69))/60}  );
\draw[dashed, thin] (.5, {sqrt(3)/6}) -- ({(45+sqrt(69))/60},{sqrt(3)*(15-sqrt(69))/60}  );
\draw[dashed, thin] (.5, {sqrt(3)/6}) -- ({(15+sqrt(69))/60},{sqrt(3)*(15+sqrt(69))/60}  );
\draw[dashed, thin] (.5, {sqrt(3)/6}) -- ({(45-sqrt(69))/60},{sqrt(3)*(15+sqrt(69))/60}  );

\filldraw[dashed,draw=black, fill=blue, opacity=.5] 
	({(52.5-sqrt(69))/30},.2  ) 
	-- ({(52.5+sqrt(69))/30},.2  ) 
	-- (1.75, {sqrt(3)/6 + .2})   
	-- cycle;
\filldraw (1.75, {sqrt(3)/6 + .2}) circle (.2pt) node[above, font=\tiny] {$e_0\equiv\left(\frac{1}{3},\frac{1}{3},\frac{1}{3}\right)$};
\filldraw (1.75,.2) circle (.2pt) node[above, font=\tiny] {$\left(\frac{1}{2},\frac{1}{2},0\right)$};
\draw[<->] ({(52.5-sqrt(69))/30},.15  ) -- ({(52.5+sqrt(69))/30},.15  );
\draw (1.75, .15) node[below, font=\tiny] {$\sqrt{4R-2}$};
\draw[->,blue] (.5,.05) .. controls +(up:.1cm) and +(left:.1cm) .. ({(52.5-sqrt(69))/30},.3  );

\filldraw[dashed, draw=black, fill=red, opacity=.5]
	({(-15-sqrt(69))/30},.2  )
	-- (-.5, {sqrt(3)/6+.2})
	-- ({(-45-sqrt(69))/60},{sqrt(3)*(15-sqrt(69))/60+.2}  )
	arc[start angle = {180+acos(5/8+sqrt(69)/24}, end angle = {180 + acos(sqrt(69)/12)}, radius = .4]
	-- cycle;
\filldraw (-.5, {sqrt(3)/6+.2}) circle (.2pt) node[above right, font=\tiny] {$e_0\equiv\left(\frac{1}{3},\frac{1}{3},\frac{1}{3}\right)$};
\draw[<->] ({(-15-sqrt(69))/30+.02},.18  ) -- node[right=.5pt,font=\tiny] {$\sqrt{R-\frac{1}{3}}$} (-.48, {sqrt(3)/6+.18});
\draw[->,red] (.22,.1) .. controls +(up:.1cm) and +(right:.1cm) .. ({-.2-sqrt(69)/30},.2);
\draw (-.5, {sqrt(3)/6+.2}) node[below=6pt, left=4pt, font=\tiny] {$\phi$};

\end{tikzpicture}

\caption{A diagram of how to divide $\Omega_{3,R}$ to determine its surface area.}
\label{fig:vol3}
\end{figure}
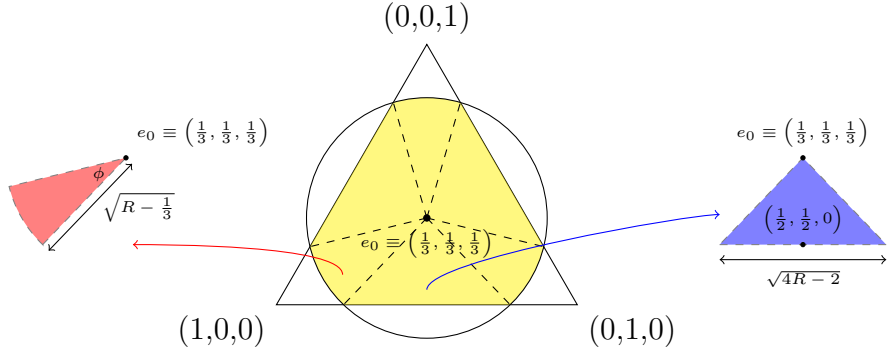

\sloppy The triangles each have area $\frac{1}{2}\frac{1}{\sqrt{6}}\sqrt{4R-2}=\frac{1}{2}\sqrt{\frac{2R-1}{3}}$. The circular sectors have radius $\delta_3(R)=\sqrt{R-1/3}$. The angle $\phi$ is the angle between the vectors $\left(\frac{1}{6}+\frac{1}{2}\sqrt{2R-1}, -\frac{1}{3} , \frac{1}{6}-\frac{1}{2}\sqrt{2R-1}\right)$ and $\left(\frac{1}{6}+\frac{1}{2}\sqrt{2R-1},  \frac{1}{6}-\frac{1}{2}\sqrt{2R-1},-\frac{1}{3} \right)$. We have

\[
\cos\phi=\frac{\frac{1}{4}(2R-1)+\frac{1}{2}\sqrt{2R-1}-\frac{1}{12}}{R-\frac{1}{3}}.
\]

The surface area of $\Omega_{3,R}$ is then

\[
\vol(\Omega_{3,R})=\frac{3}{2}\sqrt{\frac{2R-1}{3}}+\frac{3}{2}\left(R-\frac{1}{3}\right)\arccos\left(\frac{\frac{1}{4}(2R-1)+\frac{1}{2}\sqrt{2R-1}-\frac{1}{12}}{R-\frac{1}{3}}\right).
\]

The proportion of the 2-dimensional volume of $E_3$ made up by $\Omega_{3,R}$ is then 

\[
\frac{\vol(\Omega_{3,R})}{\vol(E_3)}=\sqrt{2R-1}+\sqrt{3}\left(R-\frac{1}{3}\right)\arccos\left(\frac{\frac{1}{4}(2R-1)+\frac{1}{2}\sqrt{2R-1}-\frac{1}{12}}{R-\frac{1}{3}}\right).
\]

We can bound the surface area of $\Omega_{3,R}$ from below by cutting out triangles in the corners. This is done by drawing a straight line between the intersections of the sphere and the edges with the same value in one of the coordinates. There are six places where the sphere intersects the edges of $E_3$. If we let $\theta'=\frac{1}{2}+\frac{1}{2}\sqrt{2R-1}$, these solutions are $\boldsymbol a=(\theta',1-\theta',0)$, $\boldsymbol b=(\theta',0,1-\theta')$, $\boldsymbol c=(1-\theta',0,\theta')$, $\boldsymbol d=(0,1-\theta',\theta')$, $\boldsymbol e=(0,\theta',1-\theta')$, and $\boldsymbol f=(1-\theta',\theta',0)$. Draw lines between $\boldsymbol a$ and $\boldsymbol b$, $\boldsymbol c$ and $\boldsymbol d$, and $\boldsymbol e$ and $\boldsymbol f$ (see Figure \ref{fig:approx}).

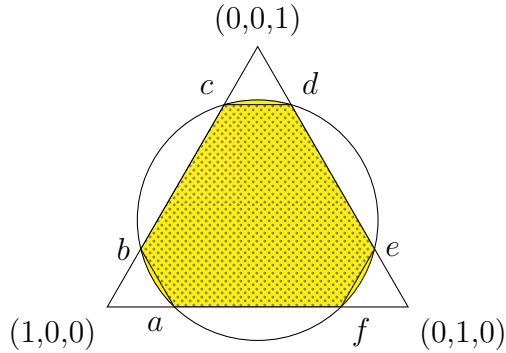
\begin{figure}[h]
\centering

\begin{tikzpicture}[scale=4]

\begin{scope}
	\clip (0,0) -- (1,0) -- (.5, {sqrt(3)/2}) -- cycle;
	\fill[yellow]   (.5, {sqrt(3)/6}) circle (.4cm);
\end{scope}

\begin{scope}
	\clip ({(15-sqrt(69))/30},0  ) --  ({(15-sqrt(69))/60},{sqrt(3)*(15-sqrt(69))/60}  ) --  ({(15+sqrt(69))/60},{sqrt(3)*(15+sqrt(69))/60}  )
	-- ({(45-sqrt(69))/60},{sqrt(3)*(15+sqrt(69))/60}  ) -- ({(45+sqrt(69))/60},{sqrt(3)*(15-sqrt(69))/60}  )
	-- ({(15+sqrt(69))/30},0  ) -- cycle;
	\draw[pattern = crosshatch dots, pattern color = gray] (0,0) -- (1,0) -- (.5, {sqrt(3)/2}) -- cycle;
\end{scope}

\draw ({(15-sqrt(69))/30},0  ) node[below left] {$a$} --  ({(15-sqrt(69))/60},{sqrt(3)*(15-sqrt(69))/60}  ) node[left] {$b$};
\draw ({(15+sqrt(69))/60},{sqrt(3)*(15+sqrt(69))/60}) node[above left] {$c$}  -- ({(45-sqrt(69))/60},{sqrt(3)*(15+sqrt(69))/60}) node[above right] {$d$};
\draw   ({(45+sqrt(69))/60},{sqrt(3)*(15-sqrt(69))/60}) node[right] {$e$}  -- ({(15+sqrt(69))/30},0  ) node[below right] {$f$};

\draw (.5, {sqrt(3)/6}) circle (.4cm) ;

\draw (0,0) node[below left] {(1,0,0)} -- (1,0) node[below right] {(0,1,0)} -- (.5, {sqrt(3)/2}) node[above] {(0,0,1)} -- cycle;
\end{tikzpicture}

\caption{How to bound the surface area of $\Omega_{3,R}$ from below using similar triangles. $\Omega_{3,R}$ is shaded as before and the lower bound area is crosshatched.}
\label{fig:approx}
\end{figure}

Due to the convexity of the sphere, the lines drawn are entirely inside the sphere. Thus the set $\Omega_{3,R}'$, which is $E_k$ without the three triangles, is entirely contained in $\Omega_{3,R}$. This gives

\begin{equation}
\vol(\Omega_{3,R}) \ge \vol(\Omega_{3,R}').
\label{eq:approx3}
\end{equation}

The triangles in the corners that are removed are equilateral triangles, with side length $\sqrt{R-\sqrt{2R-1}}$. They are thus similar to $E_3$, which is equilateral with side length $\sqrt{2}$. The ratio of the areas of the small triangles to $E_3$ is the ratio of the squared side lengths, which is $\frac{R-\sqrt{2R-1}}{2}$. This gives, finally,

\begin{equation}
\frac{\vol(\Omega_{3,R})}{\vol(E_3)}\ge 1-\frac{3R-3\sqrt{2R-1}}{2}.
\label{eq:volbound3}
\end{equation}  

\subsection{The surface area of $\Omega_{k,R}$ for $k\ge 4$}\label{sub:omegaGE4}

Calculating this surface area explicitly appears to be an open problem. We have not found a formula, but can approximate the area with a fairly sharp lower bound for certain choices of $R$. We will generalize the ``cutting off corners'' method in the $k=3$ case, which is valid for $R$ such that $\nu_{k-2}<\delta_k(R)<\nu_{k-1}$.

\begin{lemma}[The Surface Area of $\Omega_{k,R}$ for $R\in [1/2,1)$]
Let $E_k$ be the standard $k$-simplex (defined in equation \eqref{def:ek}) and $\Omega_{k,R}$ be the region $\{ \boldsymbol\theta \in E_k: |\boldsymbol\theta|^2\le R\}$. Assume that $R\in [1/2,1)$. Then 

\begin{equation}
\frac{\vol(\Omega_{k,R})}{\vol(E_k)}\ge 1-k\left(\frac{R-\sqrt{2R-1}}{2}\right)^\frac{k-1}{2}.
\label{eq:volbound}
\end{equation}

Additionally, the proportion of $E_k$ made up by $\Omega_{k,R}$ approaches 1 as $k\to\infty$.

\label{lem:sa}
\end{lemma}

\begin{proof}
Since $1/2 \le R < 1$, we have

\begin{align*}
\frac{1}{2} \le R < 1 &\Leftrightarrow \frac{1}{2} - \frac{1}{k} < R-\frac{1}{k} < 1 - \frac{1}{k}\\
&\Leftrightarrow \frac{k-2}{2k} \le R - \frac{1}{k} < \frac{k-1}{k}\\
&\Leftrightarrow \nu_{k-2}^2 \le \left[\delta_k(R)\right]^2 < \nu_{k-1}^2 \quad \text{(see equation \eqref{eq:nu})}\\
&\Leftrightarrow \nu_{k-2} \le \delta_k(R) < \nu_{k-1} 
\end{align*}

In this case, where $\nu_{k-2}\le\delta_k(R)<\nu_{k-1}$, there are intersections along the 1-dimensional edges of $E_k$ ($\boldsymbol \theta$ such that only two of its components are nonzero) with the sphere $\{\boldsymbol \theta\in \rr^k: |\boldsymbol \theta|^2=R\}$. 

Due to the convexity of the sphere, hyperplanes between these points will be contained inside the sphere. As in the $k=3$ case, we can form $k$ $(k-1)$-dimensional equilateral simplices. The $j$-th simplex will have as one of its vertices a single vertex from $E_k$ of the form $\theta_j=1$ and $\theta_i=0$ for $i\ne j$. Its remaining $k-1$ vertices will be of the form $\theta_j=\theta'$ and $\theta_i=1-\theta'$, with $i\in \{1,2,\dotsc, j-1, j+1,\dotsc, k\}$ (as in the $k=3$ case, we define $\theta'=\frac{1}{2}+\frac{1}{2}\sqrt{2R-1}$). These have edge length $\sqrt{R-\sqrt{2R-1}}$ and are similar to $E_k$ which has edge length $\sqrt{2}$. The ratio of the areas of the small simplices to $E_k$ is $\left(\frac{R-\sqrt{2R-1}}{2}\right)^\frac{k-1}{2}$. This gives the lower bound \eqref{eq:volbound}

\[
\frac{\vol(\Omega_{k,R})}{\vol(E_k)}\ge 1-k\left(\frac{R-\sqrt{2R-1}}{2}\right)^\frac{k-1}{2}.
\]

We see that $\frac{\vol(\Omega_{k,R})}{\vol(E_k)}$ approaches 1 as $k\to\infty$ as long as $\frac{R-\sqrt{2R-1}}{2}<1$. This is, in particular, true when $R\in [1/2,1)$.
\end{proof}

\subsection{Applying the volume calculations to Bayes estimators}\label{sub:estthm}

When using the prior $\dir\left(C_{k,n},\dotsc,C_{k,n}\right)$, we obtained the region \eqref{eq:riskset} where the MLE has lower risk than the Bayes estimator

\[
\left\{ \boldsymbol\theta = (\theta_1,\theta_2,\dotsc,\theta_k): \theta_i \ge 0~\forall i, \sum_{1 \le i \le k} \theta_i=1, { \sum_{1\le i \le k} \theta_i^2 \ge \frac{2n+(n+k)C_{k,n}}{2n+(k+kn)C_{k,n}}}\right\}.
\] 

The region where the \emph{Bayes estimator has lower risk than the MLE} (the complement of region \eqref{eq:riskset}) is $\Omega_{k,R}$ with $R$ defined by 

\begin{equation}
R=\frac{2n+(n+k)C_{k,n}}{2n+(k+kn)C_{k,n}}.
\label{eq:Bayesrad}
\end{equation}

We can then determine which choice of $C_{k,n}$ will yield the exponential convergence in Lemma \ref{lem:sa}. This gives the following theorem.

\begin{theorem}\label{thm:volbound}
Consider estimating the $k$-dimensional ($k\ge 3$) parameter $\boldsymbol \theta$ in the multinomial distribution based on a  simple random sample of size $n$. Under the prior $\dir\left(C_{k,n},\dotsc,C_{k,n}\right)$, the proportion of the parameter space where the Bayes estimator has lower risk than the MLE is greater than

\[
1-k\left(\frac{1}{4}\right)^{\frac{k-1}{2}} \quad\text{(for all } n\ge k), 
\]

for $C_{k,n}$ satisfying

\begin{equation}
C_{k,n}<\frac{2n}{n(k-2)-k}.
\label{eq:ckn}
\end{equation}
\end{theorem}

\begin{proof}
As noted above, the region where the Bayes estimator has lower risk than the MLE is $\Omega_{k,R}$ with $R$ defined by equation \eqref{eq:Bayesrad}

\[
R=\frac{2n+(n+k)C_{k,n}}{2n+(k+kn)C_{k,n}}.
\]

We have 

\begin{align*}
R > 1/2 &\Leftrightarrow \frac{2n+(n+k)C_{k,n}}{2n+(k+kn)C_{k,n}} > 1/2\\
&\Leftrightarrow C_{k,n}(k+2n-kn) > -2n \\
&\Leftrightarrow C_{k,n} < \frac{2n}{n(k-2)-k}.
\end{align*}

We can thus apply Lemma \ref{lem:sa}. Since $R>1/2$, we have

\begin{align*}
R > \frac{1}{2} &\Rightarrow R-\sqrt{2R-1} < \frac{1}{2}-\sqrt{\frac{2}{2}-1}\\
&\Rightarrow k\left(\frac{R-\sqrt{2R-1}}{2}\right)^\frac{k-1}{2} < k\left(\frac{1}{4}\right)^{\frac{k-1}{2}}\\
&\Rightarrow 1- k\left(\frac{R-\sqrt{2R-1}}{2}\right)^\frac{k-1}{2} > 1 - k\left(\frac{1}{4}\right)^{\frac{k-1}{2}}.
\end{align*}

Since $1 - k\left(\frac{1}{4}\right)^{\frac{k-1}{2}}\to 1$ as $k\to\infty$, we have proved that $\frac{\vol(\Omega_{k,R})}{\vol(E_k)}$ approaches 1 as $n,k\to\infty$.
\end{proof}

We can make this convergence tighter for a given $C_{k,n}$ with $R=\frac{2n+(n+k)C_{k,n}}{2n+(k+kn)C_{k,n}}>1/2$. Then we have that the proportion of the parameter space where this Bayes estimator has lower risk than the MLE is greater than 

\begin{equation}
1-k\left(\frac{R-\sqrt{2R-1}}{2}\right)^\frac{k-1}{2}.
\label{eq:volbound2}
\end{equation}

One may ask if there is a best choice of $C_{k,n}$. Note that the squared radius $R$ of the region where the Bayes estimator has lower risk ($\Omega_{k,R}$), defined in equation \eqref{eq:Bayesrad}, is a decreasing function of $C_{k,n}$. We have that $R$ approaches 1 as $C_{k,n}\searrow 0$. That is, the proportion of the parameter space where the Bayes estimator has lower risk approaches 1 (the whole space) as $C_{k,n}$ approaches 0. Note that we cannot take $C_{k,n}=0$ as the Dirichlet prior requires that $\alpha_i>0$ for all $i$. Indeed, if we could take $C_{k,n}$ to be identically zero, the Bayes estimator would be equal to the MLE! 

One could, however, use the formula for $R$ and the lower bound in \eqref{eq:volbound2} to select $C_{k,n}$ small enough to satisfy a desired level of coverage of the parameter space for a choice of $k$ (and any larger $k$). 

\begin{example} \label{ex:prior1-k}
One such choice of prior is $\dir\left(\frac{1}{k},\dotsc,\frac{1}{k}\right)$. This can be thought of as relating to using a base measure that is a probability measure, since $\sum_{1\le i \le k} \alpha_i=1$. The region where the Bayes estimator has lower risk than the MLE is $\Omega_{k,R}$ with $R$ defined by 

\[
R=\frac{2n+1+\frac{n}{k}}{3n+1} > 2/3.
\]

Thus the proportion of the parameter space where the Bayes estimator has lower risk is greater than  $1 - k\left(\frac{\frac{2}{3}-\sqrt{\frac{4}{3}-1}}{2}\right)^\frac{k-1}{2}$. Table \ref{tab:1kprior} contains estimates using \eqref{eq:Bayesrad} and 1 - \eqref{eq:volbound2}, giving an upper bound of the proportion of the parameter space where the MLE has lower risk for various values of $k$ and $n$. It also contains simulated proportions using similar methods as in section \ref{sec:sim}. The simulation used sample sizes of 10,000,000, and thus the small proportions for $k=20$ could not be detected.

\begin{table}[h]
\footnotesize
\centering
	\begin{tabular}{llll}
	\hline
	$\boldsymbol k$ & $\boldsymbol n$ & {\bf Prop (Upper Bound)} & {\bf Prop (Simulated)}\\ 
	\hline
	$k=5$ & $n=10$ & $\ee{2.68}{-3}$ & $\ee{2.12}{-3}$\\
	$k=5$ & $n=25$ & $\ee{2.95}{-3}$ & $\ee{2.32}{-3}$\\
	$k=10$ & $n=20$ & $\ee{1.97}{-6}$ & $\ee{7.00}{-7}$\\
	$k=10$ & $n=100$ & $\ee{2.30}{-6}$ & $\ee{8.00}{-7}$\\
	$k=20$ & $n=40$ & $\ee{6.53}{-13}$ & 0\\
	$k=20$ & $n=400$ & $\ee{7.88}{-13}$ & 0\\
	\hline
	\end{tabular}
\caption{Estimates of the proportion of the parameter space where the MLE has lower risk for various values of $k$ and $n=2k,k^2$. Note that it is nearly 0 for even the moderate $k=10$.}
\label{tab:1kprior}
\end{table}

\end{example}

\subsection{Simulation results for other priors}\label{sec:sim}

The requirement that $C_{k,n}<\frac{2n}{n(k-2)-k}$ precludes some priors that may be of interest. These correspond to cases with a region of interest $\Omega_{k,R}$ such that $\delta_k(R)<\nu_{k-2}$. We have not found a suitable volume lower bound for such cases. However, we have found simulation examples of a slower convergence.

\subsubsection{Uniform prior} \label{sec:uniformsim}

A common choice of prior is the uniform prior, which is the prior $\dir(1,\dotsc,1)$. Under our notation, this corresponds to $C_{k,n}=1$, which clearly does not satisfy $C_{k,n}<\frac{2n}{n(k-2)-k}$ \eqref{eq:ckn} for $k>4$.

Here the MLE has lower risk than the Bayes estimator only on the set

\begin{equation}
\left\{ \boldsymbol\theta = (\theta_1,\theta_2,\dotsc,\theta_k): \theta_i \ge 0~\forall i, \sum_{1 \le i \le k} \theta_i=1, { \sum_{1\le i \le k} \theta_i^2 \ge \frac{3n+k}{nk+2n+k}}\right\}.
\label{eq:risksetU}
\end{equation}

We used a simulation study to better understand the regions where the MLE has lower risk than the Bayes estimator under this prior. Rearranging the inequality in the region \ref{eq:risksetU}, define the function $g$

\begin{equation}
g(\boldsymbol\theta)=-3n-k+(nk+2n+k)|\boldsymbol\theta|_2^2.
\label{eq:comp}
\end{equation}

The MLE has lower risk than the Bayes estimator for $\boldsymbol\theta\in E_k$ where $g(\boldsymbol\theta)\ge 0$. 

To estimate the percent of the volume of $E_k$ where the MLE has lower risk than the Bayes estimator, we fixed $k$ and took a uniform sample of size 500,000 from $E_k$ using the R package \verb|hitandrun| \citep{hitandrun}. We then calculated $g(\boldsymbol\theta)$ for $n=k,2k,3k,4k,k^2,2k^2,3k^2,4k^2,k^3,2k^3,3k^3,4k^3,k^4,2k^4,3k^4,4k^4$ and found the percentage of the samples where $g$ is positive for each $n$. This gives a numeric estimate of the percent of the volume of $E_k$ where the MLE has lower risk than the Bayes estimator. The results are summarized in Figure \ref{fig:sampling}. Note that eventually we see that the MLE has lower risk in essentially none of the parameter space, but it is a much slower process, taking until $k=200$.

\begin{figure}[p]
\centering
\includegraphics[width=.9\textwidth]{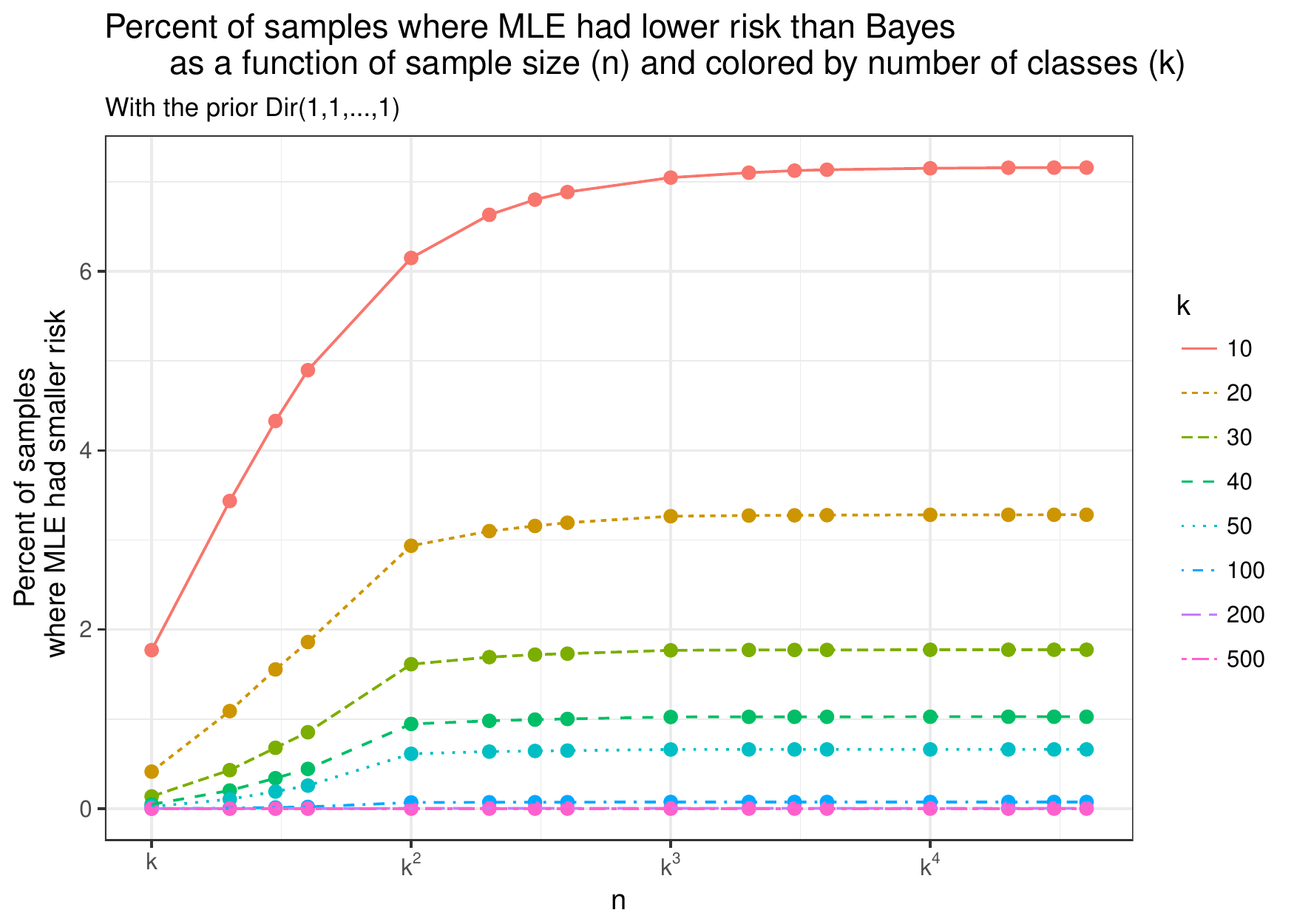}
\caption{We see that, although Theorem \ref{thm:volbound} can't be applied to the uniform prior, the MLE still has lower risk in a proportion of the parameter that decreases to 0 as $k$ increases. Note that the $n$-axis is plotted in log-scale and in terms of $k$ to make the samples comparable.}
\label{fig:sampling}
\end{figure}

\subsubsection{$\dir\left(\frac{C}{k},\dotsc,\frac{C}{k}\right)$ prior}

Note that the Dirichlet distribution on the space of all probabilities on the Borel sigma-field of a Polish space $S$ is just the Dirichlet, or multivariate Beta, distribution when $S$ is a finite set (see section~\ref{sec:tv}). By studying the limiting behavior in $n$ and $k$ of Bayes estimators in the multinomial distribution, we can hopefully gain insight into the difference between the risks of density estimation via nonparametric Bayes and using the MLE for a parametric model. However, since Ferguson's construction of the Dirichlet process prior relies on a finite base measure $\alpha(S)$, we may want to consider the sum of the prior parameters $\sum_{1\le i\le k} \alpha_i=C$, a constant, rather than $\sum_{1\le i\le k} \alpha_i=k$, which is the case for the uniform prior.

If we choose $C_{k,n}=C/k$, $R(\bb{\hat\theta},\boldsymbol\theta)\le R(\boldsymbol d_B,\boldsymbol\theta)$ (and thus the MLE has lower risk than the Bayes estimator) only on the set

\begin{equation}
\left\{ \boldsymbol\theta = (\theta_1,\theta_2,\dotsc,\theta_k): \theta_i \ge 0~\forall i, \sum_{1 \le i \le k} \theta_i=1, { \sum_{1\le i \le k} \theta_i^2 \ge \frac{2n+C+Cn/k}{2n+C+nC}}\right\}.
\label{eq:risksetC}
\end{equation}

Note that if $C=1$, we obtain the example in subsection \ref{sub:estthm}. If $C>2$, then $C_{k,n}=C/k$ does not satisfy \eqref{eq:ckn} for moderately sized $k$ and $n$. 

We again used simulation to study the regions in question. Rearranging the inequality in the region \ref{eq:risksetC}, define the function $\tilde g$

\begin{equation}
\tilde g(\boldsymbol \theta)=-Cn/k-2n-C+(2n+C+nC)|\boldsymbol\theta|_2^2.
\label{eq:comp2}
\end{equation}

The MLE has lower risk than the Bayes estimator for $\boldsymbol\theta\in E_k$ where $\tilde g(\boldsymbol\theta)\ge 0$. 

In the simulations, the limiting behavior is similar in shape to the uniform prior case, but seems to converge faster. For fixed $k$ near $C$, as $n$ increases, the percent of the volume of the parameter space where the MLE has lower risk increases to some limiting value. As $k$ increases, this limiting value decreases to zero. For $k>>C$, however, the Bayes estimator had lower risk in 100\% of the samples, indicating that the volume of the region where the MLE has lower risk is very small.

For example, with $C=30$ and $k=10, 20, 30$, the results can be found in Figure \ref{fig:sampling2}. A relatively large $C$ was chosen so that there would be several $k$ smaller than $C$ to graph.

For $k=40$, the maximum percentage was  0.12\% and for $k=50$, 0.0076\%. For $k=100, 200, 500$ (the three largest values used), the MLE had lower risk in 0\% of the samples.

\begin{figure}[p]
\centering
\includegraphics[width=1.093\textwidth]{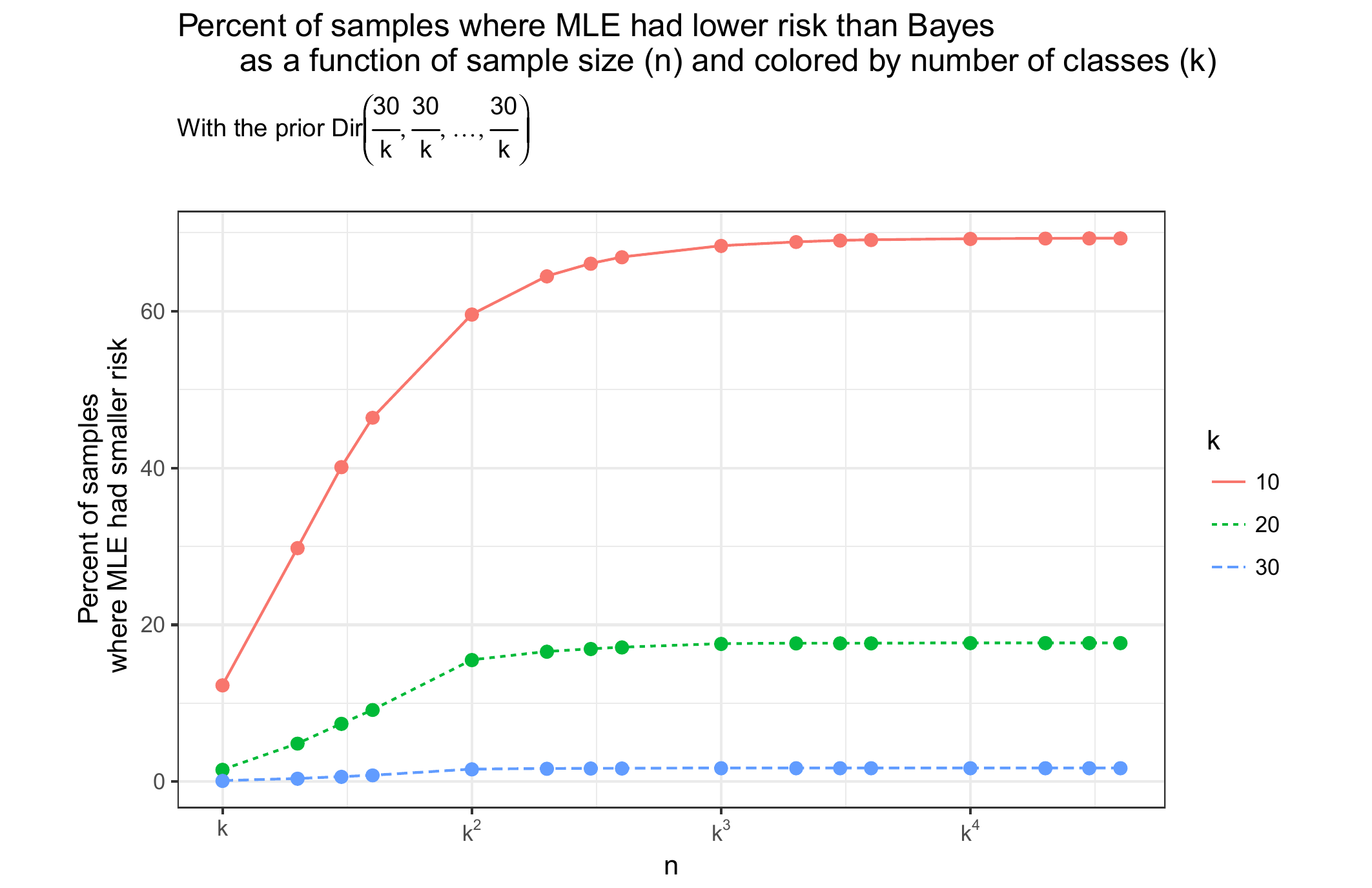}
\caption{Sampling scheme using $\tilde g$, with $C=30$. For $k=10,20,30$, the maximum percentages were 69.32\%, 17.73\%, and 1.77\%, respectively. Again we see that, although Theorem \ref{thm:volbound} does not apply, the MLE has smaller risk for almost none of the space for $k$ large enough.}
\label{fig:sampling2}
\end{figure}

Again, we see that for large enough $k$, the Bayes estimator has lower risk than the MLE for almost all of the parameter space. We would need to find a suitable volume bound to properly describe this phenomenon.

\subsubsection{$\dir(C,C,\dotsc,C)$ prior for $C>1$}

Since we have already considered the uniform prior, which is $\dir(1,1,\dotsc,1)$, it is natural to consider priors $\dir(C,C,\dotsc,C)$ with $C>1$. That is, $C_{k,n}$ is a constant rather than depending on $k$ or $n$. Priors of this type are unimodal, focusing most of their mass on the center of the parameter space $(1/k,1/k,\dotsc,1/k)$, with smaller component variances as $C$ increases.

If we choose $C_{k,n}=C$, $R(\bb{\hat\theta},\bb\theta) \le R(\bb{d_B},\bb\theta)$ (and thus the MLE has lower risk than the Bayes estimator) only on the set

\begin{equation}
\left\{ \boldsymbol\theta = (\theta_1,\theta_2,\dotsc,\theta_k): \theta_i \ge 0~\forall i, \sum_{1 \le i \le k} \theta_i=1, { \sum_{1\le i \le k} \theta_i^2 \ge \frac{2n+C(k+n)}{2n+C(k+kn)}}\right\}.
\label{eq:risksetC2}
\end{equation} 

For simulation purposes, define the function $h$, which is positive when the MLE has lower risk than the Bayes estimator, by rearranging the inequality \eqref{eq:risksetC2}:

\begin{equation}
h(\bb\theta)=-2n-c(k+n) + (2n+C(k+kn))|\bb\theta|_2^2.
\end{equation}

In simulations, a change of behavior is observed at $C=2$. For $C\in(1,2)$, the limiting behavior is similar to the uniform prior case, although with slower convergence. However, for $C\ge 2$, the opposite limiting behavior is observed. As $k$ increases, the proportion of the parameter space where the MLE has lower risk \emph{increases} to 1. See Figure \ref{fig:samplingC} for illustrative examples with $C=1.9$, $C=2$, and $C=3$.

\begin{figure}[p]
\centering
	\begin{tabular}{c}
	\includegraphics[height=.3\textheight]{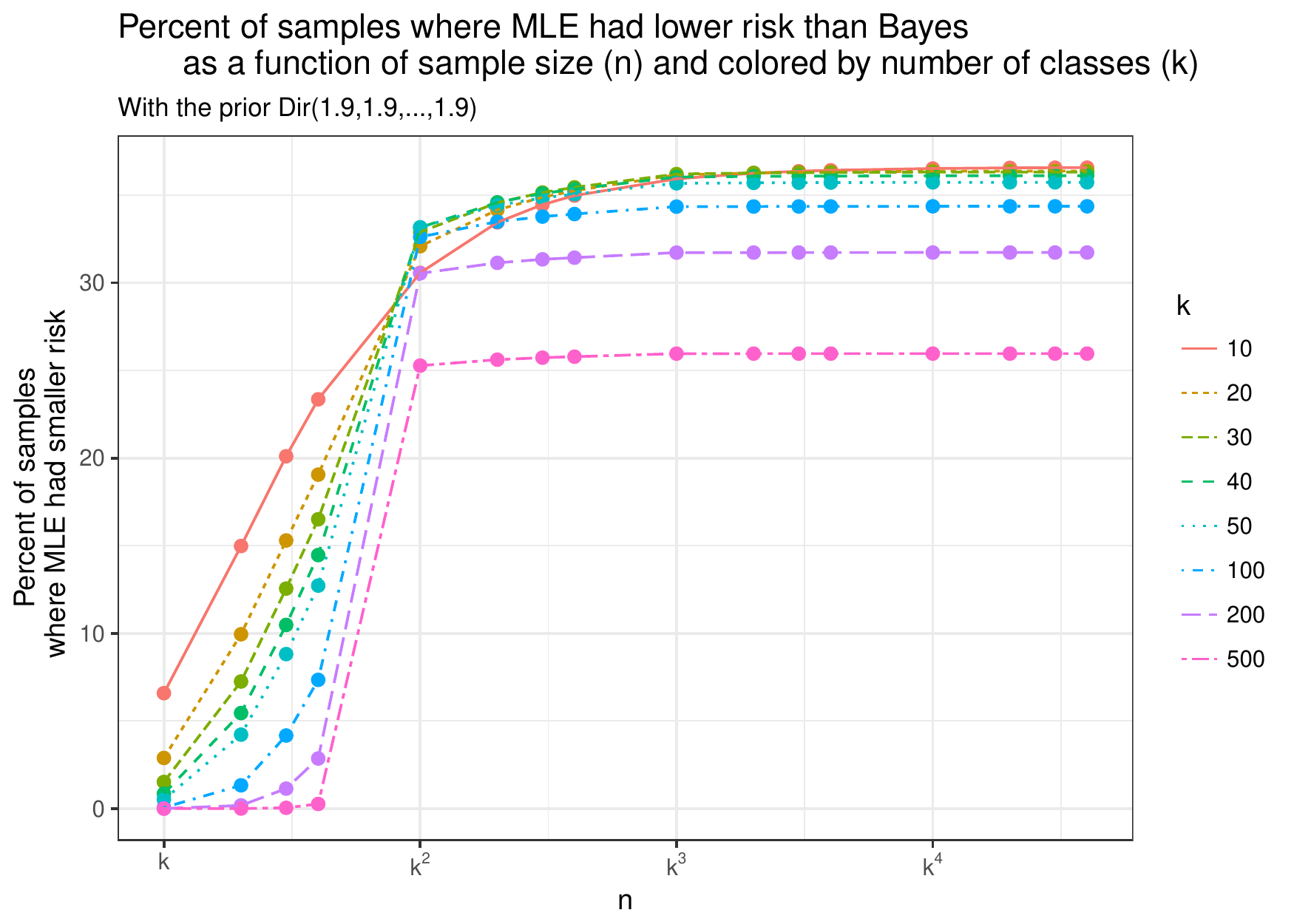} \\
	\includegraphics[height=.3\textheight]{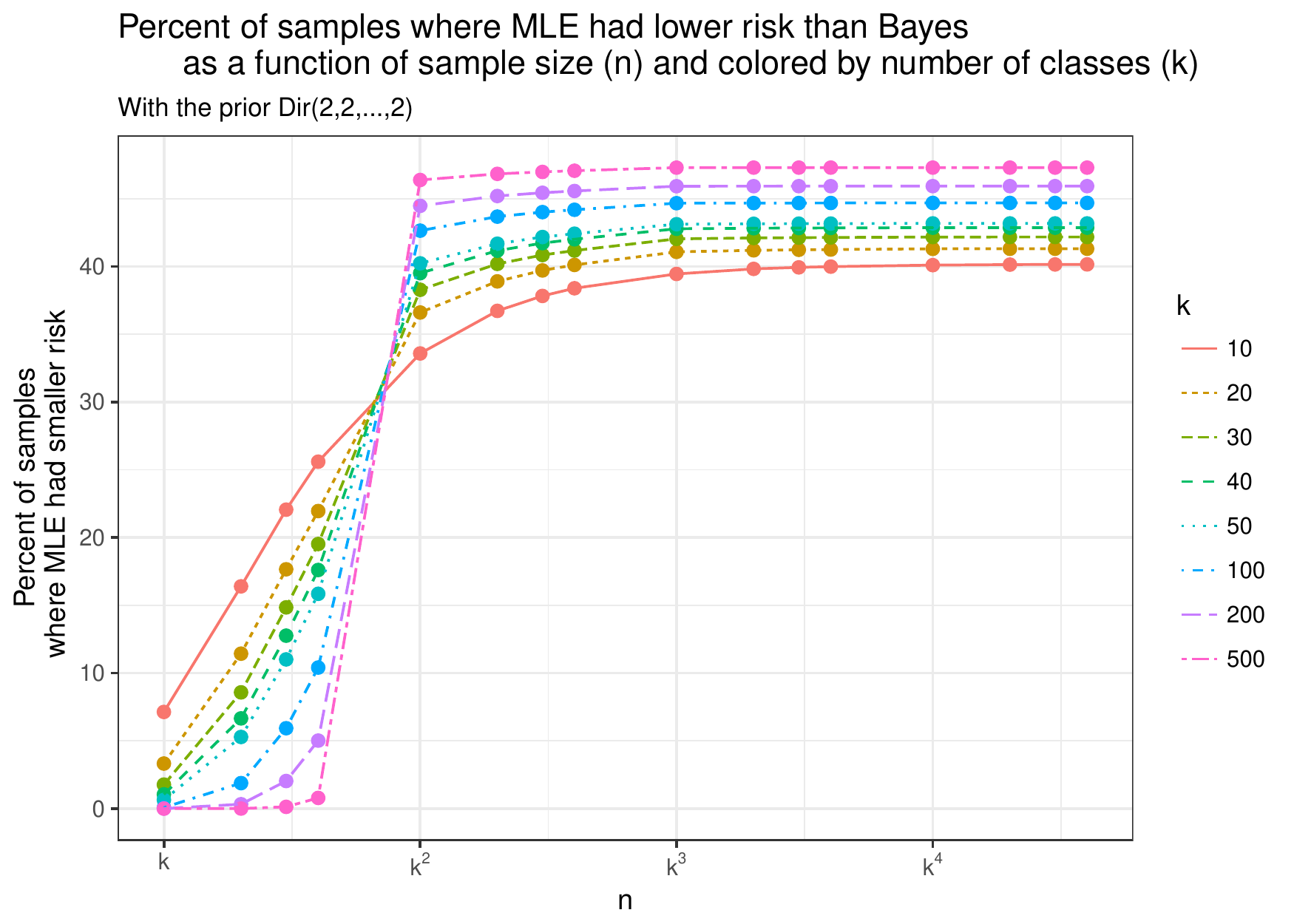} \\
	\includegraphics[height=.3\textheight]{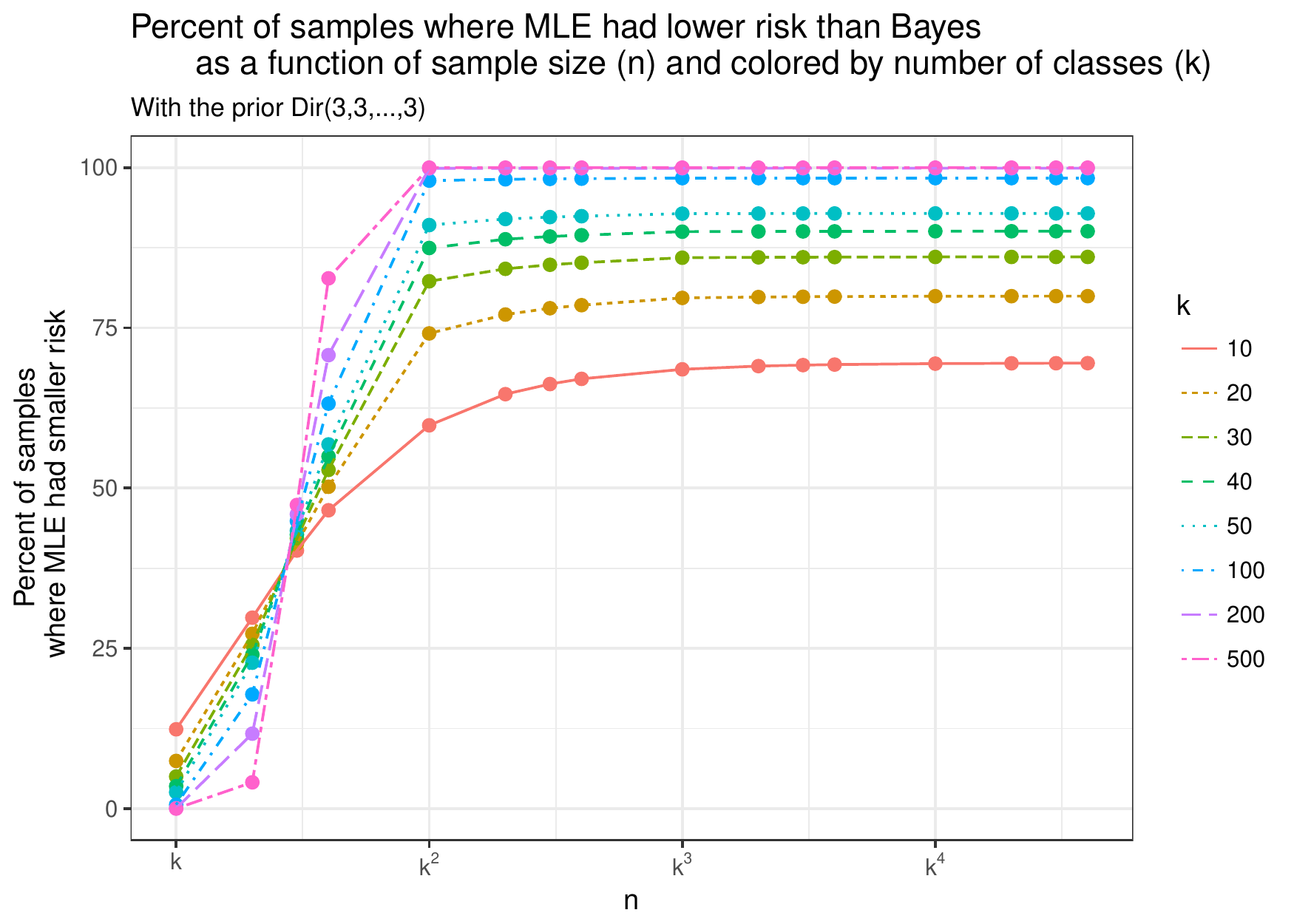} \\
	\end{tabular}
\caption{Simulation results for the prior $\dir(C,C,\dotsc,C)$ with $C=1.9, 2, 3$. Note the change in limiting behavior for $C\ge 2$, where the final percent is \emph{increasing} as $k$ increases, rather than decreasing as in the earlier examples.}
\label{fig:samplingC}
\end{figure}


\section{Average risk across the parameter space}\label{sec:avg}

In Section \ref{sec:volume}, we considered \emph{whether} the Bayes estimator had smaller risk than the MLE, and where this occurred within the parameter space. We did not consider the \emph{magnitude} of the decrease in risk. In this section, we look at the average risk (with respect to the volume measure) of the various estimators to understand the magnitude of decrease.

Recall that the risk of the MLE, by \eqref{eq:mlerisk}, is 

\[
R(\boldsymbol{\hat\theta},\boldsymbol\theta)=\frac{1-\sum_{1\le i\le k} \theta_i^2}{n}
\]

and the risk of the Bayes estimator with Dirichlet prior with $\alpha_i=C_{k,n},\;\forall i,$ is, by \eqref{eq:Bayesrisk},

\[
R(\boldsymbol d_B, \boldsymbol\theta)=\frac{n-kC_{k,n}^2}{\left(n+kC_{k,n}\right)^2}+\frac{\left(k^2C_{k,n}^2-n\right) \sum_{1\le i \le k} \theta_i^2}{\left(n+kC_{k,n}\right)^2}.
\]

Note that in each case, for fixed $k, \boldsymbol\theta$ and prior (choice of $C_{k,n}$), the risk is decreasing to 0 as $n\to\infty$. Thus any decrease in risk becomes negligible for large enough $n$.

Integrating the above over $E_k$, we obtain

\[
\int_{E_k} R(\bb{\hat\theta},\bb\theta)d\bb\theta = \frac{1}{n}A_k - \frac{1}{n}\int_{E_k}|\bb\theta|^2d\bb\theta
\]

and 

\[
\int_{E_k} R(\boldsymbol d_B, \boldsymbol\theta)d\bb\theta=\frac{n-kC_{k,n}^2}{\left(n+kC_{k,n}\right)^2}A_k+\frac{\left(k^2C_{k,n}^2-n\right)}{\left(n+kC_{k,n}\right)^2}\int_{E_k}|\bb\theta|^2d\bb\theta.
\]

By \eqref{eq:ekvol}, $A_k=\frac{\sqrt{k}}{(k-1)!}$. It can be shown that $\int_{E_k}|\bb\theta|^2d\bb\theta=\frac{2k\sqrt{k}}{(k+1)!}$. Thus we obtain that the average risks $\bar R_{\bb{\hat\theta}}$ and $\bar R_{\bb{d_B}}$ are, respectively, 

\begin{align}
\bar R_{\bb{\hat\theta}} &= \frac{\int_{E_k} R(\bb{\hat\theta},\bb\theta)d\bb\theta }{A_k}  \nonumber\\ 
&= \left(\frac{\sqrt{k}}{n(k-1)!} - \frac{2k\sqrt{k}}{n(k+1)!}\right)\div \frac{\sqrt{k}}{(k-1)!} \nonumber  \\ 
&= \frac{1}{n} - \frac{2k}{nk(k+1)}  \nonumber\\ 
&= \frac{k-1}{n(k+1)}, 
\end{align}

and

\begin{align}
\bar R_{\bb{d_B}} &= \frac{\int_{E_k} R(\boldsymbol d_B, \boldsymbol\theta)d\bb\theta}{A_k} \nonumber \\
&=\left[\frac{(n-kC_{k,n}^2)\sqrt{k}}{\left(n+kC_{k,n}\right)^2(k-1)!}+\frac{2k\left(k^2C_{k,n}^2-n\right)\sqrt{k}}{\left(n+kC_{k,n}\right)^2(k+1)!}\right]\div  \frac{\sqrt{k}}{(k-1)!} \nonumber \\
&=\frac{n-kC_{k,n}^2}{\left(n+kC_{k,n}\right)^2} + \frac{2\left(k^2C_{k,n}^2-n\right)}{\left(n+kC_{k,n}\right)^2(k+1)} \nonumber\\
&=\frac{(k-1)(kC_{k,n}^2+n)}{(kC_{k,n}+n)^2(k+1)}.
\end{align}

Then the average decrease in risk for the Bayes estimator \emph{in proportion to the average risk of the MLE} is

\begin{align}
\frac{\bar R_{\bb{\hat\theta}} - \bar R_{\bb{d_B}}}{\bar R_{\bb{\hat\theta}}} &= \left[\frac{k-1}{n(k+1)} - \frac{(k-1)(kC_{k,n}^2+n)}{(kC_{k,n}+n)^2(k+1)}\right] \div \frac{k-1}{n(k+1)} \nonumber \\
&= 1-\frac{n\left(kC_{k,n}^2+n\right)}{\left(kC_{k,n}+n\right)^2}. \label{eq:avgriskdec}
\end{align}

For fixed $k$ and $n$, this function has a global maximum at $C_{k,n}=1$, as illustrated in Figure \ref{fig:avgriskdec}. When $C_{k,n}=1$ we obtain, by plugging in to \eqref{eq:avgriskdec},

\begin{equation}
\frac{\bar R_{\bb{\hat\theta}} - \bar R_{\bb{d_{B,C_{k,n}=1}}}}{\bar R_{\bb{\hat\theta}}}=\frac{k}{k+n}=\frac{1}{1+n/k}.
\label{eq:avgriskC1}
\end{equation}

\begin{figure}[h]
\centering
\includegraphics[height=.3\textheight]{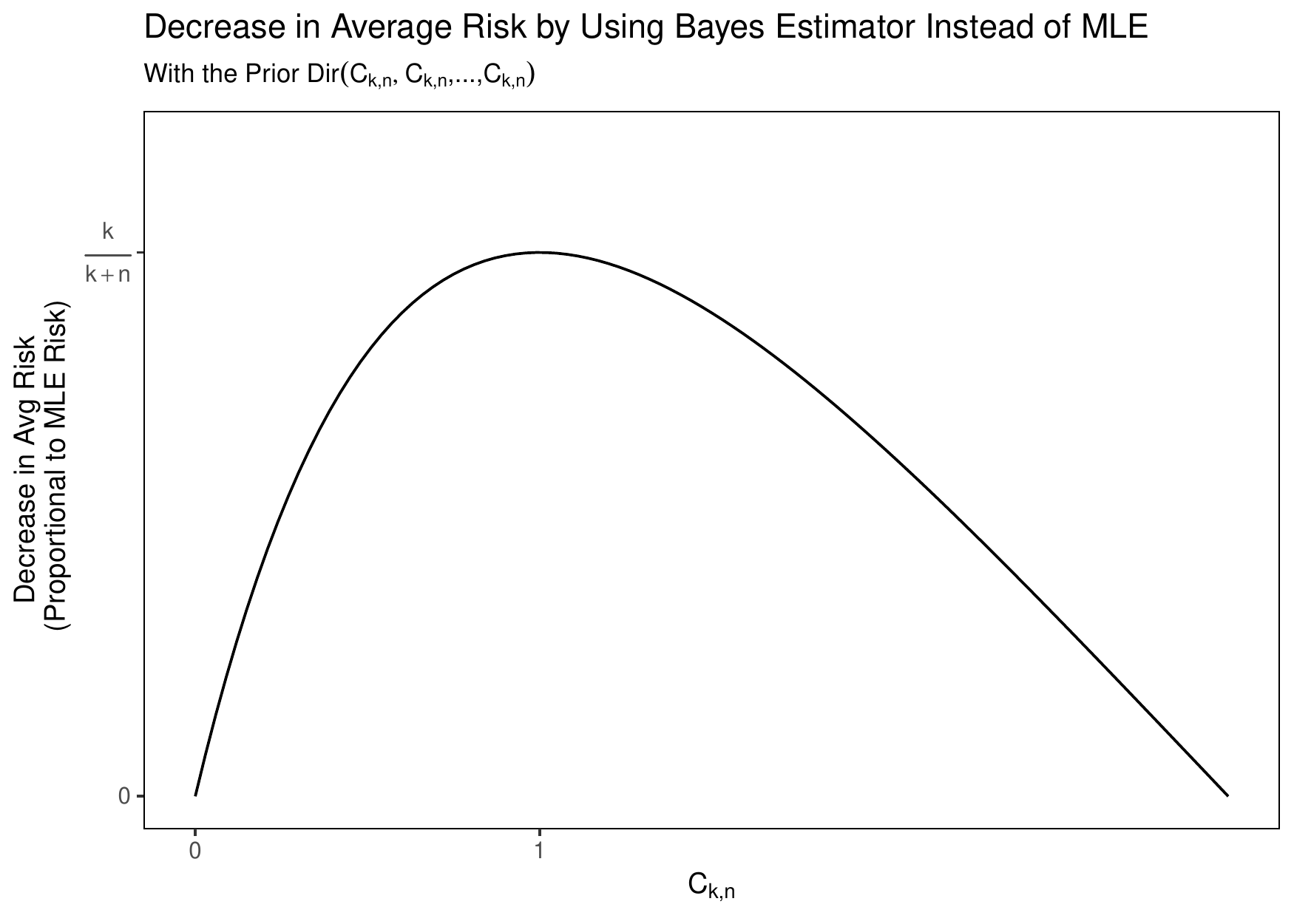}
\caption{Graph of the average decrease in risk for the Bayes estimator in proportion to the average risk of the MLE, as a function of $C_{k,n}$ with fixed $k$ and $n$. The function has a global maximum at $C_{k,n}=1$. It is postive for the region shown here, but can become negative for large enough $C_{k,n}$. This graph was made using the values $k=10$ and $n=30$; however a similar shape will result from any fixed $k$ and $n$, with a maximum at $C_{k,n}=1$ and maximum value $k/(k+n)$.}
\label{fig:avgriskdec}
\end{figure}

We see from \eqref{eq:avgriskC1} that for all $k$, there is some positive decrease in (proportional) average risk that depends on the relationship between $k$ and $n$. For example, when $n=k$, the decrease is 50\%, when $n=2k$, the decrease is 33.3\%, and when $n=k^2$, the decrease is $100\cdot\left(\frac{1}{k+1}\right)\%$. The Bayes estimator with the uniform prior is the estimator that has the smallest average risk (with respect to the Lebesgue measure) by definition (see, e.g., \citet[pp. 22-23]{textbook}). We see here that this effect, in comparison to the MLE, is most pronounced for $n$ on the order of $k$. 

On the other hand, we can consider $C_{k,n}=1/k$. This prior gave rise to an estimator that had smaller risk that the MLE for nearly the entire parameter space for even small to moderate $k$ (see example \ref{ex:prior1-k}). Plugging in to \eqref{eq:avgriskdec},

\begin{equation}
\frac{\bar R_{\bb{\hat\theta}} - \bar R_{\bb{d_{B,C_{k,n}=1/k}}}}{\bar R_{\bb{\hat\theta}}}=1 - \frac{kn^2+n}{k(n+1)^2}=1 - \frac{n(n+1/k)}{(n+1)^2}.
\label{eq:avgriskC1k}
\end{equation}

This decreases to 0 as $n$ become large, but depends less on $k$ for its limiting behavior. For the nonparametric estimation of a distribution, based on a given sample size $n$, it is useful to consider the behavior of \eqref{eq:avgriskC1k} as $k\nearrow\infty$. For fixed $n$, \eqref{eq:avgriskC1k} $\nearrow 1-\frac{n^2}{(n+1)^2}$.

Comparing the behavior for moderate to large $k$ and $n$ in \eqref{eq:avgriskC1} and \eqref{eq:avgriskC1k} shows an opposite kind of optimality than in Section \ref{sec:volume}, where the $\dir(1/k,\dotsc,1/k)$ was favored over the uniform prior. Balancing the smaller radius of the region with lower risk for the estimator under the uniform prior with its optimal decrease in average risk indicates that, for moderate to large $k$ and $n$, this estimator ($\bb{d_{B,C_{k,n}=1}}$) is preferable over other Bayes estimators. For small $k$ ($k<10$) or when the true distribution is believed to have an unknown dominating class (and thus requires an exchangeable prior with a large radius for a decrease in risk), the estimator with the $\dir(1/k,\dotsc,1/k)$ prior may be preferable. For illustration, some comparison values are in Table \ref{tab:comparison}.

\begin{table}[h]
\centering
\begingroup\footnotesize
	\begin{tabular}{ll | l | rr | rr}
	\hline
	&& {\bf MLE} & \multicolumn{1}{| l}{\bf Uniform Prior} & & \multicolumn{1}{| l}{$\boldsymbol{1/k}$ {\bf Prior}} & \\
	$\boldsymbol k$ & $\boldsymbol n$&{\bf Avg. Risk} & {\bf Decrease (\%)} & {\bf Vol. Prop.} & {\bf Decrease (\%)} & {\bf Vol. Prop.} \\
	\hline
	$k=5$		& $n=5$	& $\ee{1.33}{-1}$ & 50.00 \% & 0.9443 & 27.77 \% & 0.9977 \\
	$k=5$		& $n=25$	& $\ee{2.67}{-2}$ & 16.67 \% & 0.8926 & 6.80 \% &  0.9970 \\
	$k=10$	& $n=10$	& $\ee{8.18}{-2}$ & 50.00 \% & 0.9823 & 16.53 \% & 1\\
	$k=10$	& $n=100$ 	& $\ee{8.18}{-3}$ & 9.09 \% & 0.9385 & 1.87 \% & 1\\
	$k=50$	& $n=50$	& $\ee{1.92}{-2}$ & 50.00 \% & 0.9998 & 3.84 \% & 1\\
	$k=50$	& $n=2500$	& $\ee{3.84}{-4}$ & 1.96 \% &0.9939 & 0.08 \% & 1\\
	$k=100$	& $n=100$	& $\ee{9.80}{-3}$ & 50.00 \% & 1 & 1.96 \% & 1\\
	$k=100$ 	& $n=10000$ & $\ee{9.80}{-5}$ & 0.99 \% & 0.9993 & 0.02 \% & 1\\
	\hline
	\end{tabular}
\caption{Table comparing the average risks for the MLE and the Bayes estimators under the uniform prior and the $\dir(1/k,\dotsc,1/k)$ priors for $k=5,10,50,100$ and $n=k,k^2$. Listed as well are estimated proportions of the parameter space where the estimators have lower risk than the MLE. For the uniform prior, these are estimated using the simulation in section \ref{sec:uniformsim}. For the $\dir(1/k,\dotsc,1/k)$ prior, the estimates use \eqref{eq:Bayesrad} and \eqref{eq:volbound2}.}
\label{tab:comparison}
\endgroup
\end{table}


\section{On simulation of total variation distances between the true distribution and the distributions of (1) MLE-based frequentist estimators and (2) nonparametric Bayes estimators}
\label{sec:tv}

Let $Q$ be a probability measure on some measurable state space $(S,\mathcal S)$, which is partitioned into $k$ measurable subsets $A_1,\dotsc,A_k$. If one wishes to estimate the probabilities $Q(A_j)=\theta_j>0,\:  j=1,\dotsc,k$, based on the numbers $n_1,\dotsc,n_k$ of a random sample of size $n$ from $Q$ falling into these classes, the MLE $\boldsymbol{\hat{\theta}}=(\hat\theta_1=n_1/n,\dotsc,\hat\theta_k=n_k/n)$ is the time-honored estimate, and $(n_1,\dotsc,n_k)$ has the multinomial distribution $M(\boldsymbol n:\theta_1,\theta_2,\dotsc,\theta_k)$. One may, instead, use the Bayes estimator with the conjugate prior $\operatorname{Dir}(\alpha_1,\alpha_2,\dotsc,\alpha_k)$, namely, $\boldsymbol{d_B}=\left(\left(n_1+\alpha_1\right)/\left(n+\sum\alpha_j\right),\dotsc,\left(n_k+\alpha_k\right)/\left(n+\sum\alpha_j\right)\right)$, where $\alpha_i>0\, \forall i$. We have seen that under squared error loss, for large or moderately large $k$, $\boldsymbol{d_B}$ outperforms $\hat\theta$ on most of the parameter space in cases (1) $\alpha_i=C>0\, \forall i$ (for $C<2$), and (2) $\alpha_i=1/k\,\forall i$. For large $k$ these estimators provide approximations to $Q$. We now consider a way of providing approximations to $Q$ in variation norm for classes of $S$ with a finite volume measure $\omega$ and $Q$ absolutely continuous with respect to it.

Consider a closed bounded region $S$ such as a ball or rectangular region in an Euclidean space, or a compact Riemannian manifold such as the sphere $S^d$, Kendall's planar shape space $\Sigma_2^m$ (which is the same as the complex projective space $\mathbb CP^{m-2}$), etc., each equipped with a volume measure $\omega$. As above, we consider a partition of $S$ into $k$ subsets $A_{j,k},\: j=1,\dotsc,k,$ such that $\omega(A_{j,k})>0\,\forall j,k$ and $\omega(A_{j,k})\to 0$ as $k\to\infty$, and let $Q(A_{j,k})=\theta_{j,k}$. Let $\boldsymbol{\hat\theta}=(\hat\theta_{1,k}=n_1/n,\dotsc,\hat\theta_{k,k}=n_k/n)$ be the MLE of $\boldsymbol\theta=(\theta_{1,k},\dotsc,\theta_{k,k})$. Consider the Bayes estimator under the Dirichlet prior with $\alpha_i=1/k\,\forall i$ and under the \emph{absolute error} loss function, say $\boldsymbol{\breve\theta}=(\breve\theta_{1,1},\dotsc,\breve\theta_{k,k})$, where $\breve\theta_{j,k}$ is the median of the posterior distribution of $\theta_{j,k}$, namely, the \emph{median} of $\operatorname{Beta}(1/k+n_j,n+1-n_j-1/k),\: j=1,\dots,k$. 

The risk function of $\boldsymbol{\breve\theta}$ is 

\begin{equation}
R\left(\boldsymbol\theta,\boldsymbol{\breve\theta}\right)=\sum_{1\le j\le k} E_{\theta_{j,k}}\left|\breve\theta_{j,k}-\theta_{j,k}\right|=\sum_{j=1}^k\sum_{r=0}^n C^n_r\theta_{j,k}^r(1-\theta_{j,k})^{n-r}|F^{-1}_{(r+1/k,n-r+1-1/k)}(1/2)-\theta_{j,k}|,
\label{eq:Bayesabsrisk}
\end{equation}

where $F_{(\alpha,\beta)}$ is the distribution function of $\operatorname{Beta}(\alpha,\beta)$, $F^{-1}_{(\alpha,\beta)}$ is its inverse, and $F^{-1}_{(\alpha,\beta)}(1/2)$ is the median of $\operatorname{Beta}(\alpha,\beta)$.

The risk function of $\boldsymbol{\hat\theta}$ is

\begin{equation}
R\left(\boldsymbol\theta,\boldsymbol{\hat\theta}\right)=\sum_{1\le j\le k} E_{\theta_{j,k}}|\hat\theta_{j,k}-\theta_{j,k}|=\sum_{j=1}^k\sum_{r=0}^n C_r^n\theta_{j,k}^r(1-\theta_{j,k})^{n-r}|r/n-\theta_{j,k}|.
\label{eq:mleabsrisk}
\end{equation}

Suppose $Q$ has a continuous density $f$ (with respect to the volume measure $\omega$). We now consider the problem of estimating the approximate density of $Q$ as $f_k$, where

\begin{equation}
f_k(x)=\theta_{j,k}/\omega(A_{j,k})\quad\text{for } x\in A_{j,k} \:(j=1,\dotsc,k).
\end{equation}

Assume that $\max \{\operatorname{diam}(A_{j,k}): j=1,\dotsc,k \}\to 0$ as $k\to\infty$. Then $\int|f_k(x)-f(x)|\omega(dx)\to0$ as $k\to\infty$. We consider now the estimates of $f_k$ given by

\begin{align}
\hat f_k(x)&=\hat\theta_{j,k}/\omega(A_{j,k})\quad\text{for } x\in A_{j,k}\: (j=1,\dotsc,k),\\
\breve f_k(x)&=\breve\theta_{j,k}/\omega(A_{j,k})\quad\text{for } x\in A_{j,k}\: (j=1,\dotsc,k).
\end{align}

The $L_1$ (and thus the total variation) distances between $f_k$ and its estimates above are given by \eqref{eq:Bayesabsrisk} and \eqref{eq:mleabsrisk}. In particular, the Bayes estimator $\breve f_k(x)$ basically provides the nonparametric estimator of $f_k$ under the Dirichlet prior with base measure $\alpha_k$ on $S$ with density $\alpha_k(x)=\frac{1}{k\omega(A_{j,k})}$ for $x\in A_{j,k}\: (j=1,\dotsc,k)$.

Note that one may consider the above type of approximation for an unbounded state space by requiring that the density decays fast outside a bounded region.

Similarly, the $L_2$ distance between $f_k$ and its estimators under the squared error loss are related to the corresponding risk functions. If the $A_{j,k}$ are chosen such that $\omega(A_{j,k})=\omega(S)/k\; \forall j$ and $\hat\theta_{j,k}=n_j/n$ and $d_{j,k}=(n_j+1/k)/(n+1)$ (the Bayes estimator under squared error loss with Dirichlet prior $\operatorname{Dir}(1/k,\dotsc,1/k)$), then 

\begin{align}
\norm{f_k-\hat f_k}_2&=\sqrt{\frac{k}{\omega(S)}}\sqrt{R\left(\boldsymbol{\hat\theta},\boldsymbol\theta\right)},\label{eq:mlel2risk}\\
\norm{f_k-\tilde f_k}_2&=\sqrt{\frac{k}{\omega(S)}}\sqrt{R\left(\boldsymbol{d_B},\boldsymbol\theta\right)},\label{eq:Bayesl2risk}
\end{align}

where $R\left(\boldsymbol{\hat\theta},\boldsymbol\theta\right)$ and $R\left(\boldsymbol{d_B},\boldsymbol\theta\right)$ are as defined in \eqref{eq:mlerisk} and \eqref{eq:Bayesrisk}, respectively.

\subsection{Simulated $L_1$ distances}

We simulated these $L_1$ distances using the case where $\omega(S)=1$ by uniformly sampling 10,000 probability vectors from the standard $k$-simplex for each of $k=10, 20, 30, 40, 50, 100, 200, 500$ and calculating the risk using equations \eqref{eq:Bayesabsrisk} and \eqref{eq:mleabsrisk}. We then averaged the 100 risk calcuations to obtain an estimate of the average $L_1$ distance across the parameter space. Note that the sample size of 10,000 is smaller than that used in the $L_2$ case (1,000,000) due to increased computational complexity.

The results can be found in Tables \ref{tab:l1k10} through \ref{tab:l1k500}. Note that for large $n$, the estimated average $L_1$ distance for the Bayes estimator is often slightly larger that that for the MLE. It is unknown whether this is because the two are too close to distinguish with sample means, or that the $L_1$ distance is truly larger.

\begin{table}[p]
\begin{minipage}{.5\textwidth}
\scriptsize
\centering
\begin{tabular}{lrr}
  \hline
  $\boldsymbol{k=10}$\\
  \hline
$\boldsymbol n$ & {\bf MLE} & {\bf Bayes} \\ 
  \hline
$n=20$ & 0.4689 & 0.4541 \\ 
  $n=30$ & 0.3827 & 0.3765 \\ 
  $n=40$ & 0.3314 & 0.3283 \\ 
  $n=50$ & 0.2964 & 0.2947 \\ 
  $n=100$ & 0.2095 & 0.2095 \\ 
  $n=200$ & 0.1481 & 0.1483 \\ 
  $n=300$ & 0.1209 & 0.1211 \\ 
  $n=400$ & 0.1047 & 0.1048 \\ 
  $n=500$ & 0.09363 & 0.09374 \\ 
  $n=600$ & 0.08547 & 0.08556 \\ 
  $n=700$ & 0.07913 & 0.0792 \\ 
  $n=800$ & 0.07402 & 0.07408 \\ 
  $n=900$ & 0.06978 & 0.06984 \\ 
  $n=1000$ & 0.0662 & 0.06625 \\ 
   \hline
\end{tabular}
\caption{Simulated  $L_1$  distances for $k=10$.}
\label{tab:l1k10}
\end{minipage}
\hfillx
\begin{minipage}{.5\textwidth}

\scriptsize
\centering
\begin{tabular}{lrr}
  \hline
  $\boldsymbol{k=20}$\\
  \hline
$\boldsymbol n$ & {\bf MLE} & {\bf Bayes} \\ 
  \hline
$n=20$ & 0.682 & 0.6333 \\ 
  $n=30$ & 0.5583 & 0.5343 \\ 
  $n=40$ & 0.4839 & 0.4706 \\ 
  $n=50$ & 0.433 & 0.425 \\ 
  $n=100$ & 0.3063 & 0.3057 \\ 
  $n=200$ & 0.2166 & 0.2172 \\ 
  $n=300$ & 0.1768 & 0.1774 \\ 
  $n=400$ & 0.1531 & 0.1536 \\ 
  $n=500$ & 0.137 & 0.1373 \\ 
  $n=600$ & 0.125 & 0.1253 \\ 
  $n=700$ & 0.1157 & 0.116 \\ 
  $n=800$ & 0.1083 & 0.1085 \\ 
  $n=900$ & 0.1021 & 0.1023 \\ 
  $n=1000$ & 0.09683 & 0.09702 \\ 
   \hline
\end{tabular}
\caption{Simulated  $L_1$  distances for $k=20$.}
\end{minipage}
\end{table}

\begin{table}[p]
\begin{minipage}{.5\textwidth}
\scriptsize
\centering
\begin{tabular}{lrr}
  \hline
  $\boldsymbol{k=30}$\\
  \hline
$\boldsymbol n$ & {\bf MLE} & {\bf Bayes} \\ 
  \hline
$n=20$ & 0.8373 &  0.75 \\ 
  $n=30$ & 0.6885 & 0.6404 \\ 
  $n=40$ & 0.5978 & 0.5685 \\ 
  $n=50$ & 0.5353 & 0.5163 \\ 
  $n=100$ & 0.3791 & 0.376 \\ 
  $n=200$ & 0.2682 & 0.2687 \\ 
  $n=300$ & 0.219 & 0.2198 \\ 
  $n=400$ & 0.1896 & 0.1904 \\ 
  $n=500$ & 0.1696 & 0.1703 \\ 
  $n=600$ & 0.1548 & 0.1554 \\ 
  $n=700$ & 0.1433 & 0.1438 \\ 
  $n=800$ & 0.1341 & 0.1345 \\ 
  $n=900$ & 0.1264 & 0.1268 \\ 
  $n=1000$ & 0.1199 & 0.1203 \\ 
   \hline
\end{tabular}
\caption{Simulated  $L_1$  distances for $k=30$.}
\end{minipage}
\hfillx
\begin{minipage}{.5\textwidth}


\scriptsize
\centering
\begin{tabular}{lrr}
  \hline
  $\boldsymbol{k=40}$\\
  \hline
$\boldsymbol n$ & {\bf MLE} & {\bf Bayes} \\ 
  \hline
$n=20$ & 0.9611 & 0.8375 \\ 
  $n=30$ & 0.7948 & 0.7208 \\ 
  $n=40$ & 0.6916 & 0.6436 \\ 
  $n=50$ &  0.62 & 0.5872 \\ 
  $n=100$ & 0.4397 & 0.4325 \\ 
  $n=200$ & 0.3112 & 0.3111 \\ 
  $n=300$ & 0.2541 & 0.2549 \\ 
  $n=400$ & 0.2201 & 0.221 \\ 
  $n=500$ & 0.1968 & 0.1977 \\ 
  $n=600$ & 0.1797 & 0.1805 \\ 
  $n=700$ & 0.1664 & 0.1671 \\ 
  $n=800$ & 0.1556 & 0.1562 \\ 
  $n=900$ & 0.1467 & 0.1473 \\ 
  $n=1000$ & 0.1392 & 0.1397 \\ 
   \hline
\end{tabular}
\caption{Simulated  $L_1$  distances for $k=40$.}
\end{minipage}
\end{table}

\begin{table}[p]
\begin{minipage}{.5\textwidth}
\scriptsize
\centering
\begin{tabular}{lrr}
  \hline
  $\boldsymbol{k=50}$\\
  \hline
$\boldsymbol n$ & {\bf MLE} & {\bf Bayes} \\ 
  \hline
$n=20$ & 1.063 & 0.9081 \\ 
  $n=30$ & 0.8851 & 0.786 \\ 
  $n=40$ & 0.7723 & 0.7051 \\ 
  $n=50$ & 0.6931 & 0.6455 \\ 
  $n=100$ & 0.4926 & 0.4803 \\ 
  $n=200$ & 0.3488 & 0.3477 \\ 
  $n=300$ & 0.2849 & 0.2855 \\ 
  $n=400$ & 0.2468 & 0.2477 \\ 
  $n=500$ & 0.2207 & 0.2217 \\ 
  $n=600$ & 0.2015 & 0.2024 \\ 
  $n=700$ & 0.1865 & 0.1874 \\ 
  $n=800$ & 0.1745 & 0.1753 \\ 
  $n=900$ & 0.1645 & 0.1652 \\ 
  $n=1000$ & 0.1561 & 0.1567 \\ 
   \hline
\end{tabular}
\caption{Simulated  $L_1$  distances for $k=50$.}
\end{minipage}
\hfillx
\begin{minipage}{.5\textwidth}

\scriptsize
\centering
\begin{tabular}{lrr}
  \hline
  $\boldsymbol{k=100}$\\
  \hline
$\boldsymbol n$ & {\bf MLE} & {\bf Bayes} \\ 
  \hline
$n=20$ & 1.391 & 1.144 \\ 
  $n=30$ & 1.203 & 1.008 \\ 
  $n=40$ & 1.068 & 0.9143 \\ 
  $n=50$ & 0.9676 & 0.8454 \\ 
  $n=100$ & 0.6973 & 0.6499 \\ 
  $n=200$ & 0.4958 & 0.4838 \\ 
  $n=300$ & 0.4053 & 0.4016 \\ 
  $n=400$ & 0.3511 & 0.3503 \\ 
  $n=500$ & 0.3141 & 0.3144 \\ 
  $n=600$ & 0.2868 & 0.2876 \\ 
  $n=700$ & 0.2655 & 0.2665 \\ 
  $n=800$ & 0.2484 & 0.2495 \\ 
  $n=900$ & 0.2342 & 0.2353 \\ 
  $n=1000$ & 0.2222 & 0.2233 \\ 
   \hline
\end{tabular}
\caption{Simulated  $L_1$  distances for $k=100$.}
\end{minipage}
\end{table}

\begin{table}[p]
\begin{minipage}{.5\textwidth}
\scriptsize
\centering
\begin{tabular}{lrr}
  \hline
  $\boldsymbol{k=200}$\\
  \hline
$\boldsymbol n$ & {\bf MLE} & {\bf Bayes} \\ 
  \hline
$n=20$ & 1.652 & 1.364 \\ 
  $n=30$ & 1.512 & 1.245 \\ 
  $n=40$ & 1.392 & 1.149 \\ 
  $n=50$ & 1.291 & 1.073 \\ 
  $n=100$ & 0.9698 & 0.8479 \\ 
  $n=200$ & 0.6991 & 0.652 \\ 
  $n=300$ & 0.5732 & 0.5508 \\ 
  $n=400$ & 0.4972 & 0.4854 \\ 
  $n=500$ & 0.4451 & 0.4386 \\ 
  $n=600$ & 0.4065 & 0.4029 \\ 
  $n=700$ & 0.3764 & 0.3746 \\ 
  $n=800$ & 0.3522 & 0.3515 \\ 
  $n=900$ & 0.3321 & 0.332 \\ 
  $n=1000$ & 0.3151 & 0.3155 \\ 
   \hline
\end{tabular}
\caption{Simulated  $L_1$  distances for $k=200$.}
\end{minipage}
\hfillx
\begin{minipage}{.5\textwidth}

\scriptsize
\centering
\begin{tabular}{lrr}
  \hline
  $\boldsymbol{k=500}$\\
  \hline
$\boldsymbol n$ & {\bf MLE} & {\bf Bayes} \\ 
  \hline
$n=20$ & 1.849 & 1.543 \\ 
  $n=30$ &  1.78 & 1.483 \\ 
  $n=40$ & 1.714 & 1.425 \\ 
  $n=50$ & 1.653 & 1.369 \\ 
  $n=100$ & 1.393 & 1.152 \\ 
  $n=200$ & 1.071 & 0.9189 \\ 
  $n=300$ & 0.8932 & 0.7951 \\ 
  $n=400$ & 0.7798 & 0.7133 \\ 
  $n=500$ & 0.7003 & 0.6532 \\ 
  $n=600$ & 0.6407 & 0.6064 \\ 
  $n=700$ & 0.5941 & 0.5684 \\ 
  $n=800$ & 0.5562 & 0.5367 \\ 
  $n=900$ & 0.5248 & 0.5097 \\ 
  $n=1000$ & 0.4981 & 0.4864 \\ 
   \hline
\end{tabular}
\caption{Simulated  $L_1$  distances for $k=500$.}
\label{tab:l1k500}
\end{minipage}
\end{table}


\section{Data example: stocking jeans}\label{sec:data}

In recent years, the lack of sizing representation in clothing stores has been decried by many groups. See, for example, the article ``Women's Clothing Retailers are Still Ignoring the Reality of Size in the US'' from Quartzy \citep{shendruk_2018}. A large component of this problem is that stores do not tend to stock sizes in proportion to the distribution of clothing sizes reflected in the general population; rather, there is a notion of stocking clothing based on the ``typical customer'' for the store. This becomes a self-fulfilling prophecy, however, since choosing not to stock for parts of the population not deemed ``typical customers'' ensures that they cannot ever be customers by definition. 


Let us focus on denim jeans, which are widely considered a staple in the American woman's wardrobe. A brick-and-mortar retailer will largely only sell sizes that are currently in stock (while employees may offer to special order sizes not in stock, the majority of patrons will simply leave the store without purchasing if their size is not in stock). Since the purchasing of stock represents a risk by the retailer, it is important to accurately guess which sizes to stock. However, when taking into account both waist size and inseam, as several denim brands do, this can result in a large number of size options for stock. For example, using the Levi's online size chart and their online catalog, we calculated 59 different sizes \citep{levis}. 


The retailer could use past sales as a guide for how much of each size to stock. However, this has the effect of perpetuating errors in representation, since patrons who desired to purchase jeans but were unable since their sizes were not in stock can not be represented in the sales data. Instead, the retailer could sample the desired sizes of anyone who enters the store, regardless of whether they make a purchase. This would potentially reflect the distribution of potential customers more accuately than sales data. The retailer most likely would need to take a small to moderate sample initially since too much time with an inaccurate stock distribution may cause unrepresented segments of the population to stop coming altogether.

To simulate such a sampling scheme, we used the National Health and Nutrition Examination Survey from 2015-16 \citep{nhanes} and the Levi's size chart to estimate the true Levi's jean size distribution of adult women in the United States. After restricting to adult women and excluding those in the sample that were pregnant (as this temporarily skews waist size), there were 2697 adult women surveyed in the NHANES, with sample weighting to properly reflect the uninstitutionalized population of the United States. Using the Levi's website, we calculated 59 different jean sizes, as well as a category for those whose waist size is too high to fit into any of Levi's listed sizes (we estimated that 8.39\% of adult women in the United States fit into this category). There was one jean size that was not sampled in the NHANES. We decided that it is unlikely that this jean size does not exist in the entire population of the US, so we gave this size a proprotion equal to one half of the minimum nonzero proportion in the other sizes and then renormalized. 


We then simulated random samples of size 100 from the multinomial distribution with 60 categories using the calculated size distribution for the US adult women population. This simulates the following scenario: the retailer hopes to estimate the distribution of jean sizes his clientele desire by recording the desired jean size of a sample of 100 potential customers, and his potential custormers reflect the size distribution of the US adult female population as a whole, rather than a (potentially smaller-waisted) subpopulation. We included the 60-th category of ``no size'' since, under this scenario in which the potential customers reflect the true distribution, it is possible that customers may arrive at the store hoping to buy jeans before learning that the sizes are not large enough.

We then calculated the MLE for the size distribution by taking the sample counts and dividing by 100. We also used a uniform prior and calculated two different Bayes estimators: the estimator under squared-error ($L_2$) loss, which is the posterior mean, and the estimator under absolute-error ($L_1$) loss, which is the posterior median. In some sense, estimating under a uniform prior tries to balance between two ideas of ``fairness'': representing all sizes (a uniform prior) and representing the size distribution (the posterior mean or median given the sample). 

We repeated this simulation 1000 times. In each case, at least 18 of the size categories were unrepresented in the sample of 100. Thus the MLE estimated zero probabilities for almost one third of the sizes. On the other hand, the Bayes estimator are never zero and thus leans toward being more inclusive of sizes in stocking while still taking into account the sample data. We also calculated the $L_1$, $L_2$, and infinity (maximum) distance between the estimators and the calculated size distribution. The results are in Table \ref{tab:dist}. Note that the Bayes estimators tended to be closer that the MLE to the true size distribution by all three distance measures despite being a biased estimator. The $L_2$ Bayes estimator was closer in $L_1$ in 85.8\% of the simulations, closer in $L_2$ in 94.9\% of the simulations, and even had a smaller maximum distance (i.e. the largest absolute difference among all 60 categories) in 67.8\% of the simulations.  The $L_1$ Bayes estimator was closer in $L_1$ in 95.6\% of the simulations, closer in $L_2$ in 93.5\% of the simulations, and had a smaller maximum distance in 62.5\% of the simulations.

\begin{table}[h!]
\centering
\begin{tabular}{lrrr}
  \hline
 & $L_1$ & $L_2$ & Infinity (maximum) \\ 
  \hline
Bayes estimator ($L_2$ Loss) & 0.4599 & 0.0816 & 0.0377 \\ 
  Bayes estimator ($L_1$ Loss) & 0.4368 & 0.0830 & 0.0395 \\ 
  MLE & 0.5081 & 0.0969 & 0.0447 \\ 
   \hline
\end{tabular}
\caption[Distance between estimators and true probability]{The mean $L_1$, $L_2$, and infinity (maximum) distances between the estimators and the true size distribution in 1000 simulations.} 
\label{tab:dist}
\end{table}

In Table \ref{tab:jeans} are the estimated (true) size distribution based on the NHANES, the numbers of a stock of 1000 jeans in a Levi's store this would represent, as well as stock based on the MLE and Bayes estimators from a sample of 100 customers. Note that in three sizes, the probabilities are so small that the true size distribution still recommends to stock zero jeans in those sizes. Also note that in this particular sample 35 of the sizes were unrepresented, and thus the MLE recommended stocking less than half of the sizes. Due to rounding, the stocks are not exactly 1000. The MLE-based stock has 993, the Bayes ($L_2$) has 1006, and the Bayes ($L_1$) has 985. The $L_2$ distances between the true size distribution and the MLE, Bayes ($L_2$), and Bayes ($L_1$) estimators were, respectively, 0.122, 0.077, and 0.079. The absolute errors in the stock (the sum of all absolute differences between stock numbers in each size) for the MLE, Bayes ($L_2$) and Bayes ($L_1$) were, respectively, 687, 486, and 487. The maximum possible absolute error would be around 2000, which would occur if all 1000 pairs of jeans were stocked in the three sizes where zero stock was recommended by the true distribution (additional values are possible due to rounding).

\begin{table}[p]
\centering
\begingroup\tiny
\setlength{\tabcolsep}{6pt}
\renewcommand{\arraystretch}{.55}
\begin{tabular}{lrrrrr}
  \hline
Size (Waist.Inseam) & True $p$ & Stock (true) & Stock (MLE) & Stock (Bayes, $L_2$ Loss) & Stock (Bayes, $L_1$ Loss) \\ 
  \hline
24.28 & $\ee{1.386}{-04}$ & 0 & 0 & 7 & 5 \\ 
  24.30 & $\ee{8.657}{-04}$ & 1 & 0 & 7 & 5 \\ 
  24.32 & $\ee{2.020}{-04}$ & 0 & 0 & 7 & 5 \\ 
  25.28 & $\ee{6.930}{-05}$ & 0 & 0 & 7 & 5 \\ 
  25.30 & $\ee{1.949}{-03}$ & 2 & 0 & 7 & 5 \\ 
  25.32 & $\ee{1.000}{-03}$ & 1 & 0 & 7 & 5 \\ 
  26.28 & $\ee{6.218}{-04}$ & 1 & 0 & 7 & 5 \\ 
  26.30 & $\ee{5.900}{-03}$ & 6 & 11 & 14 & 13 \\ 
  26.32 & $\ee{4.262}{-03}$ & 5 & 0 & 7 & 5 \\ 
  27.28 & $\ee{6.986}{-04}$ & 1 & 0 & 7 & 5 \\ 
  27.30 & $\ee{1.268}{-02}$ & 14 & 0 & 7 & 5 \\ 
  27.32 & $\ee{8.139}{-03}$ & 9 & 11 & 14 & 13 \\ 
  27.34 & $\ee{1.214}{-03}$ & 1 & 0 & 7 & 5 \\ 
  28.28 & $\ee{1.356}{-03}$ & 1 & 0 & 7 & 5 \\ 
  28.30 & $\ee{5.628}{-03}$ & 6 & 0 & 7 & 5 \\ 
  28.32 & $\ee{1.305}{-02}$ & 14 & 0 & 7 & 5 \\ 
  28.34 & $\ee{2.675}{-03}$ & 3 & 0 & 7 & 5 \\ 
  29.28 & $\ee{1.847}{-03}$ & 2 & 0 & 7 & 5 \\ 
  29.30 & $\ee{1.181}{-02}$ & 13 & 46 & 34 & 37 \\ 
  29.32 & $\ee{1.471}{-02}$ & 16 & 0 & 7 & 5 \\ 
  29.34 & $\ee{3.006}{-03}$ & 3 & 0 & 7 & 5 \\ 
  30.28 & $\ee{2.008}{-03}$ & 2 & 0 & 7 & 5 \\ 
  30.30 & $\ee{1.203}{-02}$ & 13 & 23 & 21 & 21 \\ 
  30.32 & $\ee{1.350}{-02}$ & 15 & 0 & 7 & 5 \\ 
  30.34 & $\ee{4.348}{-03}$ & 5 & 0 & 7 & 5 \\ 
  31.28 & $\ee{4.102}{-03}$ & 4 & 0 & 7 & 5 \\ 
  31.30 & $\ee{2.713}{-02}$ & 30 & 11 & 14 & 13 \\ 
  31.32 & $\ee{3.483}{-02}$ & 38 & 0 & 7 & 5 \\ 
  31.34 & $\ee{9.686}{-03}$ & 11 & 0 & 7 & 5 \\ 
  32.28 & $\ee{3.122}{-03}$ & 3 & 0 & 7 & 5 \\ 
  32.30 & $\ee{3.314}{-02}$ & 36 & 46 & 34 & 37 \\ 
  32.32 & $\ee{4.211}{-02}$ & 46 & 69 & 48 & 52 \\ 
  32.34 & $\ee{1.271}{-02}$ & 14 & 34 & 27 & 29 \\ 
  33.28 & $\ee{3.128}{-03}$ & 3 & 0 & 7 & 5 \\ 
  33.30 & $\ee{2.157}{-02}$ & 24 & 11 & 14 & 13 \\ 
  33.32 & $\ee{3.553}{-02}$ & 39 & 23 & 21 & 21 \\ 
  33.34 & $\ee{3.100}{-03}$ & 3 & 11 & 14 & 13 \\ 
  34.28 & $\ee{9.037}{-03}$ & 10 & 0 & 7 & 5 \\ 
  34.30 & $\ee{4.477}{-02}$ & 49 & 46 & 34 & 37 \\ 
  34.32 & $\ee{7.274}{-02}$ & 79 & 103 & 68 & 76 \\ 
  34.34 & $\ee{1.153}{-02}$ & 13 & 0 & 7 & 5 \\ 
  16W.S &$\ee{ 8.729}{-03}$ & 10 & 0 & 7 & 5 \\ 
  16W.M &$\ee{ 9.149}{-03}$ & 10 & 0 & 7 & 5 \\ 
  16W.L & $\ee{1.929}{-03}$ & 2 & 0 & 7 & 5 \\ 
  18W.S & $\ee{4.694}{-02}$ & 51 & 80 & 55 & 60 \\ 
  18W.M &$\ee{ 5.803}{-02}$ & 63 & 92 & 62 & 68 \\ 
  18W.L & $\ee{8.997}{-03}$ & 10 & 11 & 14 & 13 \\ 
  20W.S & $\ee{4.353}{-02}$ & 48 & 0 & 7 & 5 \\ 
  20W.M & $\ee{4.290}{-02}$ & 47 & 46 & 34 & 37 \\ 
  20W.L & $\ee{8.351}{-03}$ & 9 & 0 & 7 & 5 \\ 
  22W.S & $\ee{3.617}{-02}$ & 39 & 69 & 48 & 52 \\ 
  22W.M & $\ee{3.962}{-02}$ & 43 & 57 & 41 & 45 \\ 
  22W.L & $\ee{8.016}{-03}$ & 9 & 34 & 27 & 29 \\ 
  24W.S & $\ee{2.187}{-02}$ & 24 & 0 & 7 & 5 \\ 
  24W.M & $\ee{3.911}{-02}$ & 43 & 103 & 68 & 76 \\ 
  24W.L & $\ee{7.272}{-03}$ & 8 & 11 & 14 & 13 \\ 
  26W.S & $\ee{1.851}{-02}$ & 20 & 34 & 27 & 29 \\ 
  26W.M & $\ee{2.248}{-02}$ & 25 & 11 & 14 & 13 \\ 
  26W.L & $\ee{2.546}{-03}$ & 3 & 0 & 7 & 5 \\ 
  No size & $\ee{8.393}{-02}$ & 0 & 0 & 0 & 0 \\ 
   \hline
\end{tabular}
\endgroup
\caption[Size distribution (NHANES) and stock]{The true size distribution based on NHANES 15-16 as well as the stock of about 1000 based on the truth and estimators from a sample of size 100. The total absolute errors of the three estimated stocks are 687, 486, and 487, respectively.} 
\label{tab:jeans}
\end{table}


\section{Final remarks}\label{sec:final}

In this article it is shown that in a multinomial model with a moderately large number of $k$ cells and even a reasonable large sample size, the Bayes estimator with a multivariate Beta (or Dirichlet) $\dir(C_{k,n},\dotsc,C_{k,n})$ prior has a smaller risk under squared error loss than that of the MLE on most of the parameter space, for the cases $C_{k,n}=1$, and $C_{k,n}=1/k$. The volume of domination is larger for the case $C_{k,n}=1/k$ than $C_{k,n}=1$. When compared by average performance this domination over the MLE persistes, but is more pronounced for the uniform prior than for the case $C_{k,n}=1/k$. Simulation studies also show the surprising fact that for $C_{k,n}\ge 2$ the performance of the Bayes estimator rapidly declines.

The choice $C_{k,n}=1/k$ is motivated by the fact that it provides a simple approximation of the nonparametric estimation of an unknown distribution with a Dirichlet process prior \`a la \citet{ferguson1973}. In this context, there have been some simulation studies where nonparametric Bayes procedures have been found to outperform frequentist ones. A dramatic example may be found in \citet{david1}. Here a random sample is drawn from a parametric distribution on Kendall's planar shape space with density $f_0=f(\mathord{\cdot},\theta_0)$ with a given parameter value $\theta=\theta_0$, and three estimates of $f_0$ are compared: $\hat f=f(\mathord{\cdot},\hat\theta)$ with $\hat\theta$ as the MLE of the parameter $\theta$, the standard nonparametric kernel density estimator $g$ of $f$, and a nonparametric Bayes estimator $h$ of $f$. One would expect that the asymptotically efficient MLE of a correctly specified parametric model would perform better than its nonparametric competitors. But, surprisingly, a set of 20 simulations each with a fresh random sample of size 200 show the following average $L_1$-distances $d$: (1) $d(h,f_0)=0.44$, (2) $d(g,f_0)=1.03$, (3) $d(\hat f,f_0)=0.75$. Although the present study points to the superiority of the Bayes procedure compared to frequentist ones such as the histogram method, the differences do not appear to be that dramatic. Perhaps the method of representing an unknown density as a mixture of an appropriate parametric family and estimating the mixture by Ferguson’s Dirichlet process, as used by \citet{david1} should be preferred (also see \citet{ghosh_ram2003} and \citet{ghosal2017}). Still the present article provides a simple and widely applicable Bayes estimation of a nonparametric distribution, which perhaps may be sharpened to be more effective. 

\clearpage
\bibliography{references}

\end{document}